\documentclass{article}

\usepackage[english]{babel}
\usepackage{graphicx} 
\usepackage{amsmath, amsthm, amsfonts, amssymb, amscd}
\usepackage[all]{xy}
\usepackage[a4paper]{geometry}
\usepackage{enumerate}
\usepackage{url}
\usepackage{indentfirst}
\usepackage{latexsym}
\usepackage{mathrsfs}
\usepackage{hyperref}
\usepackage{tikz}
\usepackage{subcaption}
\usepackage{caption}
\usepackage{placeins}
\usepackage{pst-func}
\usepackage{array}
\newcolumntype{M}[1]{>{\centering\arraybackslash}m{#1}}
\newcolumntype{N}{@{}m{0pt}@{}}

\usepackage[all]{xy}
\usepackage{enumerate}
\usepackage{indentfirst}

 \def\Hom{\operatorname{Hom}}
 \def\coker{\operatorname{Coker}}

 \def\Div{\operatorname{div}}  
  \def\Supp{\operatorname{Supp}}

  \def\der{\operatorname{Der}}  

 \def\Aut{\operatorname{Aut}}

\def\spec{\operatorname{Spec}}
\def\Inv{\operatorname{Inv}}

\newcommand{\Ou}{\mathcal O}

 \newcommand{\Py}{\mathbb P}
  
\newcommand{\then}{\Rightarrow}

\newcommand{\m}{\mathfrak{m}}

\usepackage{amsmath,amssymb,amsthm}

\newtheorem{theorem}{thm}[section]
\newtheorem{thm}[theorem]{Theorem}
\newtheorem{lemm}[theorem]{Lemma}     
\newtheorem{cor}[theorem]{Corollary}
\newtheorem{defn}[theorem]{Definition}
\newtheorem{prop}[theorem]{Proposition}
\newtheorem{exm}[theorem]{Example}
\newtheorem{Remark}[theorem]{Remark}

\usepackage[english]{babel}    

\usepackage{lipsum}
\usepackage{kantlipsum}

\title{Non-smooth regular curves via a descent approach}

\author{G. Borelli, C. D. D. Moreira and R. Salomão}

\begin{document}


\bibliographystyle{plain}
\maketitle

\begin{abstract}
This paper aims to continue the classification of non-smooth regular curves, but over fields of characteristic three. These curves were originally introduced by Zariski as generic fibers of counterexamples to Bertini's theorem on the variation of singular points of linear series. Such a classification has been introduced by Stöhr, taking advantage of the equivalent theory of non-conservative function fields, which in turn occurs only over non-perfect fields $K$ of characteristic $p>0$. We propose here a different way of approach, relying on the fact that a non-smooth regular curve in $\Py^n_K$ provides a singular curve when viewed inside $\Py^n_{K^{1/p}}$. Hence we were naturally induced to the question of characterizing singular curves in $\Py^n_{K^{1/p}}$ coming from regular curves in $\Py^n_K$. To understand this phenomenon we consider the notion of integrable connections with zero $p$-curvature to extend Katz's version of Cartier's theorem for purely inseparable morphisms, where we solve the above characterization for the slightly general setup of coherent sheaves. Moreover, we also had to introduce some new local invariants attached to non-smooth points, as the differential degree. As an application of the theory developed here, we classify complete, geometrically integral, non-smooth regular curves $C$ of genus $3$, over a  separably closed field $K$ of characteristic $3$, whose base extension $C \times_{\spec K}{\spec \overline{K}}$ is  non-hyperelliptic with normalization having geometric genus $1$.
\end{abstract}

\vspace{1cm}

\section{Introduction}

Bertini's theorem (or Bertini-Sard theorem) on variable of singular points in linear systems is no longer true when we pass from a field of characteristic zero to a field of positive characteristic, as announced by Zariski in \cite{Z}. More precisely, he found fibrations by singular curves between smooth varieties, or equivalently, regular (in the sense of points with regular local rings) and non-smooth (meaning that the usual Jacobian criterion is not satisfied after extending the base field to its algebraic closure) curves over non-perfect fields, playing as their generic fibers. A basic example, working in every positive characteristic, is the fibration from $\mathbb{A}^2=\spec k[x,y]$ to  $\mathbb{A}^1=\spec k[t]$, given by the assignment $f(x,y)=y^2+x^p$, where $k$ is an algebraically closed field of characteristic $p>0$. Moreover, its generic fiber is the regular and non-smooth curve $y^2+x^p-t$  over $K=k(t)$.

Such fibrations are often associated with interesting positive characteristic phenomena, suggesting that their classification is an interesting problem to tackle. For instance, the extension of Enriques' classification of surfaces to positive characteristic (see \cite{BM}),  counterexamples to Kodaira's vanishing theorem (see \cite{Zh} and \cite{Muk}), a class of isolated hypersurface singularities having infinite Milnor number (see \cite{HRS}) and coverings of $\Py^2$ by families of singular, strange and non-classical curves (see \cite{Salomon} and \cite{hilario2024fibrations}).

Taking advantage of that ``non-smooth regular curves'' and ``non-conservative function fields'' reflect the phenomenon of genus droop and are equivalent objects, many authors started to use the literature about the second one (\cite{Bedoya} and \cite{Sthor2}) to classify the first one (cf.   \cite{Stohr} and many subsequent papers of Sthör's students). 
Roughly speaking the main idea for the classification that appears in the literature so far can be divided into two steps. The first one is similar to the classification of elliptic curves and consists of finding a basis of certain vector spaces of global sections associated with some divisors, like obtaining a Weierstrass form. The second one consists in finding a basis of the space of regular/holomorphic differential forms to construct an embedding from such curves in a projective space, when they are non-hyperelliptic. Actually, such an embedding can be obtained after a base extension of the field of constants to its algebraic closure.  

In this work we propose a different way to approach this phenomenon by using techniques similar to Galois descent theory, that is, we use the category of schemes over algebraic extensions of a field $K$ to deduce properties of non-smooth regular curves defined over $K$. To do this, we will restrict ourselves to the classification of geometrically integral, non-smooth regular curves $C$ over a separably closed field $K$ of positive characteristic $p$. The starting point for understanding the beginning of our approach is the ``local version of Kimura's theorem'' (see Corollary 3.2 in \cite{Sthor2}), which says that $C$ is a non-smooth regular curve if and only if the extended curve $C^{K^{1/p}}:=C\times_{\spec K} \spec K^{1/p}$ is non-regular. In this way, our strategy is to classify geometrically integral, non-smooth regular curves $C$ embedded in the projective space $\Py^n_{K}$, by classifying geometrically integral and non-regular (or singular) curves in  $\Py^n_{K^{1/p}}$ which are obtained from base extensions of regular curves in  $\Py^n_{K}$.

 To do this we have an important issue to overcome, because an embedding on a subvariety of $\Py^n_{K^{1/p}}$ may not be a base extension of a subvariety of $\Py^n_{K}$. For instance, $K^{1/p}$-rational points that are not $K$-rational points.  
 Thus, we need to use an appropriate language to approach this phenomenon, which will be the category of $\Py^n_{K}$-invariant sheaves over a geometrically integral $\Py^n_{K}$-subschemes of $\Py^n_{K^{1/p}}$.

We then have three main goals/steps in this work, which are detailed in the following.
\begin{enumerate}
    \item To establish and explore the categorical language of  $X$-invariant subschemes of $Y$, where $Y\to X$ is a purely inseparable morphism, factorizing the absolute Frobenius morphism $F\colon X \to X^{(1)}$, between geometrically integral varieties over $K$ with localy free sheaf of relative Kähler differentials $\Omega_{Y/X}$;
    \item The study of local and global properties at singular points lying over non-smooth and regular points of algebraic varieties over $K$;
    \item The classification of complete, geometrically integral, non-smooth regular curves $C$ of genus $3$ over a non-perfect and separably closed field $K$, where $C^{\overline{K}}:=C \times_{\spec K} \spec \overline{K}$ is a non-hyperelliptic curve of geometric genus $1$.
\end{enumerate}

According to the notations of step 1, Section \ref{chapter2} is devoted to the characterization of subvarieties of $Y$ arising from pulling-back subvarieties of $X$, which will be briefly sketched below.

In Subsection \ref{section invariant sheaves} we study a somewhat more general question than the one treated above. Indeed, we investigate $\Ou_Y$-coherent sheaves $\mathcal{S}$ coming from base extensions $\mathcal{T}\otimes_{\Ou_X} \Ou_Y$ of $\Ou_X$-coherent sheaves $\mathcal{T}$. For this purpose it is helpful to use the notion of integrable connections with $p$-curvature zero 
$\nabla:\mathcal{S}\to \Omega_{Y/X}\otimes_{\Ou_Y}\mathcal{S}$ introduced by Katz in \cite{Katz} (See Definition \ref{def integrable connections}), which induces a $\Ou_X$-coherent sheaf  $\mathcal{S}^{\nabla}=\{x\in \mathcal{S} \, | \, \nabla(x)=0\}$
(See Definition \ref{kernel of connections}). In the case where $Y$ is a smooth $X$-scheme, Katz's version of Cartier's Theorem (Theorem 5.1 in \cite{Katz}) says that the $\Ou_Y$-coherent sheaves $\mathcal{S}$, coming from base extensions of $\Ou_X$-coherent sheaves, correspond precisely to $\mathcal{S}^{\nabla}\otimes_{\Ou_X} \Ou_Y$, for some connection $\nabla$. On the other hand, in our situation, $Y$ is not required to be a smooth $X$-scheme anymore, as we can see in Lemma \ref{pi1 is not smooth}, but we are able to show an analogue version of Cartier's Theorem, which can be found in Theorem \ref{Funtor igualdad de ida} and Theorem \ref{funtor igualdad de vuelta}.

In Subsection \ref{invariant subvarieties of Y} we take advantage of the isomorphism of functors in Theorem \ref{Funtor igualdad de ida} and Theorem \ref{funtor igualdad de vuelta} to characterize
ideal sheaves of $Y$ obtained from base change of ideal sheaves of $X$ (see Proposition \ref{Ideais Invarantes}) and, consequently, subvarieties of $Y$ coming from pull-back of subvarieties of $X$. It is interesting to rephrase the theory developed in Subsection \ref{invariant subvarieties of Y} in the language of foliation theory. As we can see in Proposition 1.9 of \cite{Miy}, there is a $p$-closed foliation $\mathcal{F}$ such that $X=Y^{\mathcal{F}}$, that is, $\Ou_X=\operatorname{Ann}(\mathcal{F})=\{a \in \Ou_Y\,;\, D (a)=0 \text{ for every } D \in \mathcal{F}\}$. Therefore, Proposition \ref{Ideais Invarantes} says that $X$-invariant subvarieties of $Y$ correspond exactly to subvarieties of $Y$ which are invariant by $\mathcal{F}$.   

Our second aim in this work is the study of local invariants of curves in order to support the obtaining of normal forms of curves in $\Py^n_{K^{1/p}}$, coming from curves in $\Py^n_{K}$, with prescribed data. In this way,  Section \ref{local invariants at non-smooth points} is dedicated to the study of local invariants associated at non-smooth  points $P$ of  complete and geometrically integral curves $C$, over a separably closed field $K$, and at the respective singularity $P^{\overline{K}}$ of $C^{\overline{K}}:=C\times_{\spec K} \spec\overline{K}$. For instance, in Subsection \ref{subsection 3.1} we relate in Proposition \ref{pre grado emb} the embedding dimension at $P^{\overline{K}}$ and the dimension of the relative Kähler differentials $\Omega_{\textbf{k}(C,P)/K}$, where $\textbf{k}(C,P)$ is the residue field at $P$. Moreover we also give a criterion in Proposition \ref{semigrupo de un punto regular como se ve} to know when the non-smooth point $P$ is regular.

In Subsection \ref{section 3.2} we introduce the differential degree at a regular point $P$ of $C$, which is defined by
\[d(C,P)=1+\dim_{\overline{K}} \coker (\alpha)\]
where  
$\alpha:\Ou^{\overline{K}}_P\Omega_{ \overline{K}\Ou_P /\overline{K}}\to  \Omega_{\Ou^{\overline{K}}_P/\overline{K}}$ is the natural homomorphism,
$\Ou_P$ is the local ring of $C$ at $P$, $\Ou^{\overline{K}}_P$ is the normalization of the local ring $\overline{K}\Ou_P$ in $F^{\overline{K}}=\overline{K}F$ - that is, it is exactly the local ring at the unique point  $\widetilde{P^{\overline{K}}}$ of the normalization of $C^{\overline{K}}$ lying over $P$ -
 and 
$\Omega_{\Ou/\overline{K}}$ is the module of K\"ahler differentials of the $\overline{K}$-algebra $\Ou$ over $\overline{K}$. We highlight here that the differential degree is a new invariant measuring the smoothness at a point, as we can see in Remark \ref{differential degree and smoothness 1} and Proposition \ref{rL=1 imply smooth}.  

By using such invariant, we compute in Theorem \ref{descomposicion simple del semigrupo} the semigroup $\Gamma_{C^{\overline{K}},P^{\overline{K}}}=\{l\in \mathbb{Z}| l=\nu(x) \mbox{ for some } x\in \overline{K}\Ou_{P} \}$ at the singularity $P^{\overline{K}}$ of $C^{\overline{K}}$, where $\nu \colon F^{\overline{K}} \to \mathbb{Z}$ is the discrete valuation with valuation ring $\Ou^{\overline{K}}_P$. Indeed, we show in a suitable situation that $\Gamma_{C^{\overline{K}},P^{\overline{K}}}=d(C,P)\mathbb{Z} + p\mathbb{Z}$. Consequently, we conclude that the singularity degree (or $\delta$-invariant) 
$\delta_{P^{\overline{K}}}:=\dim_{\overline{K}}\widetilde{\overline{K}\Ou_{P}}/\overline{K}\Ou_{P}$ at $P^{\overline{K}}$ is equal to $\frac{(d(C,P)-1)(p-1)}{2}$. 

We also show in Subsection \ref{section 3.2} that the degree at $P$ divides $c(C,P)+d(C,P)-1$ where $c(C,P)$ is the conductor of the semigroup $\Gamma_{C,P}=\{l\in \mathbb{Z}| l=\nu(x) \mbox{ for some } x\in \Ou_{P} \}$ at $P$ (see Theorem \ref{cota para o grau do ponto}). We notice that it provides a sharp upper bound for the degree at $P$, as we can see in the case classified in Section \ref{section 6}.   

Still in step $2$, Section \ref{Sectionn 4} is dedicated to the study of properties of sheaves of a regular and non-smooth curve $C$. Following the central spirit of this work we find, in Theorem \ref{Lema a cerca de igualdad de divisores}, a curve (the normalization of a base extension of $C$) and a morphism to $C$ in which isomorphisms between invertible sheaves of $C$ are guaranteed by isomorphisms between their pullbacks. Moreover, we also compare in Theorem \ref{segundo lema de divisores canonicos} three standard sheaves  of $C$ in order to characterize its extrinsic geometry after embedding $C$ in a projective space over $K$.

To support our strategy for the classification done in this work, we restrict the characteristic to three and we study a geometric structure data associated at $K^{1/3}$-rational points of $\Py^2_{K^{1/3}}$ in Section \ref{Sectionn 5}. In this way we define the $1$-type attached to these points, which is the minimum degree of an $\Py^2_{K}$-invariant irreducible curve passing through them. In Proposition  \ref{tipo 1 tipo 2 prueba} we prove, under certain suitable hypothesis on the residue field at these points, that the  $1$-type is always $1$ or $2$. By using some invertible sheaves, associated with the intersection between plane curves and the quotients of cubic powers of lines in $\Py^2_K$ - bringing us back to the notion of angles between curves - we can characterize up to automorphisms of $\Py^2_{K^{1/3}}$, who are these points and the respective $X$-invariant curves of minimum degree through them (see Proposition \ref{proposicion que ahora si resuelve el problema de los puntos}). In particular we use this result to characterize geometrically when a binary system of $1$-type two points, associated to non-smooth points of a regular plane quartic, belong to the same $\Py^2_K$-invariant conic.

The last section of this work is dedicated to the implementation of the theory constructed in the last three sections to give a different approach to the classification of complete, geometrically integral, non-smooth regular curves $C$ of genus $3$ over a non-perfect and separably closed field $K$ with $C^{\overline{K}}$ being non-hyperelliptic.  If $p_g(C)$ is the geometric genus of $C$ and $p$ is the characteristic of $K$, then the upper bound proved by Tate in \cite{tate1952genus}, says that $p\leq 2 p_g(C) +1$. Consequently, the phenomenon proposed to be classified here may occur in characteristic smaller or equal to seven. The cases of characteristic seven and five were developed, respectively, by Stichtenoth in \cite{stichtenoth1978konservativitat} and by Stöhr and Villela in \cite{stohr1990non}. The advances in characteristic three and two were obtained by Salomão in \cite{Salomon} and \cite{Salomon2} and by Stöhr and Hilario in \cite{hilario2023fibrations} and \cite{hilario2024fibrations}, all of them assuming $p_g(C^{\overline{K}})=0$. In this paper we will treat the case  $p_g(C^{\overline{K}})=1$ - where $C$ will be called by a geometrically elliptic curve - focusing in the specific case of characteristic three.

After using the comparison of sheaves in Theorem \ref{segundo lema de divisores canonicos} and the obtaining of a $\Py^2_K$-invariant conic in $\Py^2_{K^{1/3}}$ through the points $P^1$ and $Q^1$ (lying over the non-smooth points $P$ and $Q$ of $C$), as discussed above, we conclude Subsection \ref{subsection 61} by showing that $C$ is defined by $L_1L^3-F^2$, as a curve embedded in $\Py^2_{K}$, where $F, L \in K[x,y,z]$ is a conic and a line, respectively, and $L_1 \in K^{1/3}[x,y,z] \setminus K[x,y,z]$ is a line, with $P, Q \in F\cap L_1^3$ and $P,Q \not\in L$ (see Proposition \ref{Ecuacion de C en genero 1}).

In Subsection \ref{subsection 5.2}, we allow the action of some $\Py^2_K$-automorphisms in the equation $L_1L^3-F^2$ to conclude that $C$ is isomorphic to a plane curve living in one of the following families of quartics (see Theorem \ref{Primer Teorema central de este trabajo}).

\begin{itemize}
\item[{\bf$\mathcal{C}_0 \, :$}] $a(x^2-yz)^2+((t_1^2-t_2^2)x+(t_2-t_1)y+(t_1t_2^2-t_2t_1^2)z)^3x$ where $t_1\neq t_2\in K^{1/3}\setminus K$ and $a\in K\setminus\{-(t_1^2-t_2^2)^3\}$.
\item[{\bf$\mathcal{C}_1 \, :$}] $a(x^2-yz)^2+((t_1^2-t_2^2)x+(t_2-t_1)y+(t_1t_2^2-t_2t_1^2)z)^3z$ where $t_1\neq t_2\in K^{1/3}\setminus K$ and $a\in K$.
\item[{\bf$\mathcal{C}_2 \, :$}]
$a(xy)^2+(t_1x+t_2y-t_1t_2z)^3z$ where $t_1\neq t_2\in K^{1/3}\setminus K$  and $a\in K$.
\end{itemize}

Moreover, each curve belonging to one of these families is a geometrically elliptic curve of genus $3$ as stated in Theorem \ref{tercer Teorema central de este trabajo}. We finish this work studying (in Theorem \ref{tabla de isomorfismos}) when two curves in $\mathcal{C}_0 \cup \mathcal{C}_1 \cup \mathcal{C}_2$ are equivalent via an automorphism of $\Py_{K}^2$. Indeed, it is precisely what we need to study when two geometrically elliptic curves of genus $3$ are isomorphic, because, if two such curves are isomorphic, then their canonical classes must be preserved and, consequently, we will have an isomorphism between the global sections of both canonical sheaves.

\section{Invariant objects under purely inseparable base changes.} \label{chapter2}

Let  $K$ be a non-perfect and separably closed field of characteristic $p>0$. As we have seen earlier, the characterization of subvarieties of $\Py^n_{K^{1/p}}$ coming from pulling-back (or base extensions) of subvarieties of $\Py^n_{K}$ is a relevant problem in our context. We may see in Example \ref{pi1 is not smooth} that our situation is a particular setup of the following one, which will be assumed for the rest of the section. Let $X$ and $Y$ be two varieties defined over $K$ and $L$, respectively, where $L|K$ is a finite and purely inseparable field extension satisfying: there is a purely inseparable morphism $\pi\colon Y \to X$, factorizing the absolute Frobenius morphism of $Y$, such that $\Omega_{Y/X}$ is locally free. Therefore, this section aims to characterize subvarieties of $Y$ coming from pulling-back of subvarieties of $X$.

\subsection{The category of $X$-invariant sheaves over $Y$} \label{section invariant sheaves}

In this subsection we will characterize $\Ou_Y$-coherent sheaves coming from base extensions of $\Ou_X$-coherent sheaves. Such characterization follows the idea of Cartier's Theorem (cf. Theorem 5.1.1 in \cite{Katz}) with the difference that in our situation the morphism $\pi\colon Y \to X$ is not required to be smooth, as we can see in the next example. Before that we must recall the definition of a $p$-basis of a field extension $K_1|K$, which is a set of elements $s_1,\dots, s_r$ in $K_1$ satisfying the following two conditions: 
\begin{itemize}
    \item 
$(K_1)^p(K,s_1,\dots,s_r)=K_1$; 
\item There are derivations $\delta_1,\dots,\delta_r\in \der_K(K_1)$ such that $\delta_i(s_j)=\delta_{ij}$, where $\delta_{i,j}$ is the Kronecker delta and  $\der_K(K_1)$ is the $K$-module of $K$-derivations of $K^{1/p}$.
\end{itemize}

As a consequence of Theorem 26.5 in \cite{Mats}, we may see that there are $p$-bases for the field extensions $L|K$ and $K(X)|K$.

\begin{exm} \label{pi1 is not smooth}
 Suppose that $X$ is a geometrically integral scheme  over $K$, $Y=X\times \spec(K^{1/p})$ and $\pi$ is the morphism of base change. Then: \begin{enumerate}
     \item The purely inseparable morphism  $\pi \colon Y \to X$ factorizes the absolute Frobenius morphism of $Y$ and is non-smooth.
     \item Every derivation $\delta\in \der_K(K^{1/p})$ can be uniquely extended to a derivation $D_{\delta}$ in  $H^0(Y,\der_{\Ou_X}(\Ou_Y))$.
     \item $\Omega_{Y/X}$ is a free $\Ou_Y$-coherent sheaf.
 \end{enumerate} 
\end{exm}
\begin{proof}
\begin{enumerate}
    \item Naturally, we only need to care about the non-smoothness of $\pi$. After applying some basic properties of fiber products we may see that the fiber $Y_x$ over each $x\in X$ is given by   
$\spec \left( K^{1/p}\otimes_{K}\kappa(x) \right)$
where $\kappa(x)$ is the residue field of $x$. If $s_1,\ldots,s_r$ is a $p$-basis of $\kappa(x)|K$ then $\kappa(x)$ is given by $L[s_1,\ldots,s_r]$ (with $L=K\kappa(x)^p$). Since $K \subset \kappa(x) \subset K^{1/p^n}$ for some $n$, we have that $\kappa(x)=K\kappa(x)^{p^n}[s_1,\ldots,s_r]$ and that $Y_x$ is given by the ideal
$\left\langle X_1^{p^n}-s_1^{p^n},\ldots, X_r^{p^n}-s_r^{p^n}\right\rangle$ of $K^{1/p}[X_1,\ldots,X_r]$  
 where $X_1,\ldots,X_r$ are algebraically independent over $K^{1/p}$. Hence $Y_x\times_{\spec K^{1/p}} \spec \overline{K}$ is given by the ideal  $\left\langle (X_1-s_1)^{p^n},\ldots, (X_r-s_r)^{p^n}\right\rangle$ of $\overline{K}[X_1,\ldots,X_r]$ and, consequently,
 is singular.

    \item Let $s_1,\ldots, s_r$ be a $p$-basis of $K^{1/p}$ over $K$ and let $t_1,\dots, t_n$ be a $p$-basis of $K(X)$ over $K$. Since $X$ is geometrically integral then $X$ is geometrically reduced and, consequently, $K(Y)|K^{1/p}$ is separable in the sense of \S{26} in \cite{Mats}.  As $X$ is geometrically integral then $K(X)$ and $K^{1/p}$ are linearly disjoint over $K$.  So $s_1,\dots, s_r,t_1,\dots, t_n$ is a $p$-basis of $K(Y)|K$. Given $D_{\delta}\in \der_K(K(Y))$ satisfying $D_{\delta}(s_i)=\delta(s_i)$ and $D_{\delta}(t_i)=0$, we have  $(D_{\delta}){\mid_{K^{1/p}}}=\delta$ and $D(K(X))=0$. If $D'$ is another extension of $\delta$ such that $D'(K(X))=0$, then $D_{\delta}(s_i)=D'(s_i)$ for $i=1,\ldots, r$ and $D_{\delta}(t_j)=D'(t_j)=0$ for $j=1,\ldots, n$, implying that $D_{\delta}=D'$.

If $U$ is an affine open subset of $X$ and  $V=\pi^{-1}(U)$ then $\Ou_Y(V)=\Ou_X(U)\otimes_K K^{1/p}$. But $D(\Ou_X(U))=0$, together with Leibniz's rule, implies that $D_{\delta}(\Ou_Y(V))\subset \Ou_Y(V)$. Therefore $D_{\delta}\in \der_{\Ou_X(U)}(\Ou_Y(V))$. Since it is true for every affine open subset $V$ in $Y$, then $D\in H^0(Y, \der_{\Ou_X}(\Ou_Y))$.

\item Since $\Ou_Y(V)=\Ou_X(U)\otimes_K K^{1/p}$ then $\Omega_{Y/X}(V)$ is generated by $ds_1,\ldots,ds_r$, which are also global sections in $H^0(Y,\Omega_{Y/X})$. Take $\delta_1,\ldots,\delta_r$ the dual basis of $s_1,\ldots,s_r$ and $D_{\delta_i}$ the extension of $\delta_i$ to $\Ou_Y$, which factors into $d:\Ou_Y\to \Omega_{Y/X}$ followed by a $\Ou_Y$-morphism $\widetilde{D}_{\delta_i}\colon \Omega_{Y/X} \to  \Ou_Y$. If $x=\sum a_ids_i=0$ with $a_i\in \Ou_Y(V)$, then $\widetilde{D}_{\delta_j}(x)=a_j=0$, so $ds_1,\ldots,ds_r$ form a basis of $\Omega_{Y/X}(V)$.

\end{enumerate}

\end{proof}

Now we are going toward obtaining the characterization announced at the beginning of the subsection. It will be done using connections of sheaves as the principal ingredient and we need to define what it means.

Given an open subset $U$ of $X$, $V=\pi^{-1}(U)$ and a derivation
 $D$ of $\Ou_Y(V)$ vanishing on $\Ou_X(U)$, it follows from the universal property of $\Omega_{Y/X}(V)$, that $D$ factors into $d:\Ou_Y(V)\to \Omega_{Y/X}(V)$ followed by a $\Ou_Y$-morphism $D_d:\Omega_{Y/X}(V) \to \Ou_Y(V)$.

\begin{defn} \label{definition of connections}
A $X$-{\it connection} on a $\Ou_Y$-coherent sheaf $\mathcal{S}$ is a $\Ou_X$-morphism  $\nabla:\mathcal{S}\to \Omega_{Y/X}\otimes_{\Ou_Y}\mathcal{S}$ satisfying
\[\nabla(ax)=d(a)\otimes x+ a\nabla(x)\]
 for every $a\in \Ou_Y(V)$ and $x\in \mathcal{S}(V)$, where $V$ is any open subset of $Y$.  
Moreover, if $\nabla:\mathcal{S}\to \Omega_{Y/X}\otimes_{\Ou_Y}\mathcal{S}$ is a $\Ou_X$-morphism and  $D \in \der_{\Ou_X(U)} (\Ou_Y(V))$,  then we define $\nabla_{D}\in\Hom_{\Ou_X} (\mathcal{S}(V),\mathcal{S}(V))$ by $\nabla_{D} = (D_d\otimes_{\Ou_Y} 1_{\mathcal{S}})\circ\nabla$, where $1_{\mathcal{S}}$ is the identity of $\mathcal{S}$,  $U$ is any open subset of $X$ and $V=\pi^{-1}(U)$.  
\end{defn}


\begin{lemm}\label{igualdad de conexiones}
Let us consider $\mathcal{S}$ be a $\Ou_Y$-coherent sheaf,  $\nabla:\mathcal{S}\to \Omega_{Y/X}\otimes_{\Ou_Y}\mathcal{S}$ be a $\Ou_X$-morphism and $y\in \mathcal{S}(V)$, where $V$ is an open subset of $Y$ such that $\Omega_{Y/X}(V)$ is free. If $\nabla_{D}(y)= 0$ for every derivation $D\in  \der_{\Ou_X(U)}(\Ou_Y(V))$ then $\nabla(y)=0$, where $U=\pi(V)$. 
\end{lemm}
\begin{proof}
Suppose that $\Omega_{Y/X}(V)$ is free. Then there is a $\Ou_Y(V)$-basis $ds_1,\dots,ds_r$ of $\Omega_{Y/X}(V)$ and there is a basis $D_1,\dots,D_r$  of $\der_{\Ou_X(U)}(\Ou_Y(V))$ such that $D_i(s_j)=\delta_{ij}$. So $\nabla(y)=ds_1\otimes y_1+\ldots+ ds_r\otimes y_r$ for some $y_1,\ldots, y_r\in \mathcal{S}(V)$. Therefore 
$0=\nabla_{D_i}(y)=\sum_j  ((D_i)(s_j) \cdot y_j)= y_i$ implies the result.  
\end{proof}

\begin{defn} \label{kernel of connections}
If $\mathcal{S}$ is a $\Ou_Y$-coherent sheaf and $\nabla:\mathcal{S}\to \Omega_{Y/X}\otimes_{\Ou_Y} \mathcal{S}$ is a $X$-connection, then we define the $\Ou_X$-coherent sheaf $\mathcal{S}^{\nabla}$ by $$\mathcal{S}^{\nabla}(U)=\{x\in \mathcal{S}(\pi^{-1}(U)) \, | \, \nabla(x)=0\}$$ for every  open subset $U$ of $X$.
\end{defn}

\begin{lemm}\label{Calculo de conexxiones}
Let $\mathcal{S}$ be a $\Ou_Y$-coherent sheaf, $\nabla$ be a $X$-connection on $\mathcal{S}$ and $D\in \der_{\Ou_X(U)}(\Ou_Y(V))$, where $U$ is an open subset of $X$ and $V=\pi^{-1}(U)$. If $\nabla_{D}(x)=0$ and $a\in \Ou_Y(V)$, then $\nabla_{D}(ax)=D(a)x$.
\end{lemm}
\begin{proof}
We have that $\nabla_{D}(ax)=(D_d\otimes_{\Ou_Y} 1_{\mathcal{S}})\circ \nabla(ax)=(D_d\otimes_{\Ou_Y} 1_{\mathcal{S}})(da\otimes x)=D(a)x$.
\end{proof}
\begin{defn}\label{def integrable connections}
We say that a $X$-connection $\nabla:\mathcal{S}\to \Omega_{Y/X}\otimes_{\Ou_Y} \mathcal{S}$ on  a $\Ou_Y$-coherent sheaf $\mathcal{S}$  is integrable if for every $D,D'\in \der_{\Ou_X(U)}(\Ou_Y(V))$ (where $U$ is an open subset of $X$ and $V=\pi^{-1}(U)$) we have that \[  \nabla_{[D,D']}-[\nabla_{D},\nabla_{
D'}]=0.\]
Moreover, we say that $\nabla$ has $p$-curvature zero when
\[
\nabla_{D^{p}}-(\nabla_{D})^p=0. 
\]
When $\nabla$ is integrable with $p$-curvature zero, we will say that $\nabla$ is $X$-invariant.
\end{defn}

\begin{defn}\label{definicion MIC}
We denote by $\Inv(Y/X)$ the category whose objects are pairs $(\mathcal{S},\nabla)$ where $\mathcal{S}$ is a $\Ou_Y$-coherent sheaf and $\nabla$ is an $X$-invariant connection on $\mathcal{S}$. 

Let $(\mathcal{S}_1,\nabla_1)$ and $(\mathcal{S}_2, \nabla_2)$  be two pairs in  $\Inv(Y/X)$. We say that a morphism $f\colon \mathcal{S}_1\to \mathcal{S}_2$ belongs to the category $\Inv(Y/X)$ when 
$$f\circ (\nabla_1)_{D}= (\nabla_2)_{D}\circ f$$ for any derivation $D\in \der_{\Ou_X(U)}(\Ou_Y(V))$ and any open subset $U$ of $X$, where $V=\pi^{-1}(U)$.
\end{defn}
\begin{Remark}
Notice that any $f:\mathcal{S}_1 \to \mathcal{S}_2$ in $\Inv(Y/X)$ restricts itself to a $\Ou_X$-morphism from $(\mathcal{S}_1)^{\nabla_1}$ to $(\mathcal{S}_2)^{\nabla_2}$.
\end{Remark}

The next two propositions provide useful examples of elements in $\Inv(Y/X)$.

\begin{prop}
\label{d integrable}
The pair $(\Ou_Y,d)$ belongs to  $\Inv(Y/X)$ where $d\colon \Ou_Y\to \Omega_{Y/X}$ is the natural $\Ou_X$-derivation. 
\end{prop}
\begin{proof}
Clearly $d:\Ou_Y\to \Omega_{Y/X}$ is a $X$-connection satisfying $\Ou_Y^{d}=\Ou_X$.
Let  $D,D'\in \der_{\Ou_X(U)}(\Ou_Y(V))$, where $U$ is an open subset of $X$  and $V=\pi^{-1}(U)$. Lemma \ref{Calculo de conexxiones} says that $d_{D}=D$ and $d_{D'}=D'$. Thus $(d_{D})^p=D^p=(d_{D^p})$ and  $(d_{[D,D']})=[D,D']=[d_{D},d_{D'}]$, concluding the $X$-invariance of $d$. 
\end{proof}

\begin{Remark}
If $\mathcal{S}_1$ is a $\Ou_X$-coherent sheaf and $\mathcal{S}=\Ou_Y\otimes_{\Ou_X}\mathcal{S}_1$, then we can identify $\Omega_{Y/X}\otimes_{\Ou_Y} \mathcal{S} \simeq \Omega_{Y/X}\otimes_{\Ou_X} \mathcal{S}_1$. Therefore the $\Ou_X$-morphism   $\nabla= d\otimes_{\Ou_X} 1_{\mathcal{S}_1}$ can be seen as a $\Ou_X$-morphism from $\mathcal{S}$ to $\Omega_{Y/X}\otimes_{\Ou_Y} \mathcal{S}$.
\end{Remark}

\begin{prop}\label{definicion del funtor}
If $\mathcal{S}_1$ is a $\Ou_X$-coherent sheaf, $\mathcal{S}=\Ou_Y\otimes_{\Ou_X}\mathcal{S}_1$ and $\nabla= d\otimes_{\Ou_X} 1_{\mathcal{S}_1}$, then $(\mathcal{S},\nabla)$  belongs to the category $\Inv(Y/X)$.
\end{prop}
\begin{proof}
Let $U$ be an open subset of $X$ and $V=\pi_1^{-1}(U)$. From hypothesis, any $x\in \mathcal{S}(V)$ can be written by $x=\sum b_i\otimes y_i$ with $b_i\in \Ou_Y(V)$ and $y_i\in \mathcal{S}_1(U)$. Thus for $a\in \Ou_Y(V)$ we obtain $\nabla(a x)=\nabla(\sum ab_i\otimes y_i)=\sum d(ab_i)\otimes y_i=da\sum b_i\otimes y_i+a\sum d(b_i)\otimes y_i=(da)x+a\nabla(x)$, which says that   $\nabla$ is a $X$-connection.  Since $d$ is $X$-invariant then the same happens for $\nabla$ and therefore $(\mathcal{S},\nabla)$ belongs to the category $\Inv(Y/X)$. 
\end{proof}
\begin{prop}\label{morphism in MIC} 
Suppose that $\mathcal{S}_1$ and $\mathcal{T}_1$ are $\Ou_X$-coherent sheaves and $f_1:\mathcal{S}_1\to \mathcal{T}_1$ is a $\Ou_X$-morphism. If $f=1_{\Ou_Y}\otimes_{\Ou_X}f_1 $ is the $\Ou_Y$-morphism from $\mathcal{S}=\Ou_Y \otimes_{\Ou_X} \mathcal{S}_1$ to $\mathcal{T}=\Ou_Y\otimes_{\Ou_X} \mathcal{T}_1$, then it belongs to the category $\Inv(Y/X)$.
\end{prop}
\begin{proof}
By definition we need to prove that $f\circ(\nabla_{\mathcal{S}})_{D}=(\nabla_{\mathcal{T}})_{D}\circ f$ for every
 $D\in \der_{\Ou_X(U)}(\Ou_Y(V))$ and every open subset $U$ of $X$, where $V=\pi^{-1}(U)$,  $\nabla_{\mathcal{S}}=d\otimes_{\Ou_X}1_{\mathcal{S}_1}$ and $\nabla_{\mathcal{T}}=d\otimes_{\Ou_X}1_{\mathcal{T}_1}$. Since each $x\in \mathcal{S}(V)$ can be written by $x=\sum a_i\otimes x_i$ where $x_i\in \mathcal{S}_1(U)$ and $a_i\in \Ou_Y(V)$, it follows that $f((\nabla_{\mathcal{S}})_{D}(x))= f\left(\sum D(a_i)\otimes x_i\right)=\sum D(a_i)\otimes f_1(x_i)=(\nabla_{\mathcal{T}})_{D}(f(x))$  
 for every $D\in \der_{\Ou_X(U)}\Ou_Y(V)$, proving the proposition.
\end{proof}

Now we define the functor $\mathfrak{F}$, from the category $\Inv(Y/X)$ to the category of coherent  $\Ou_X$-coherent sheaves,  and the functor $\mathfrak{G}$, going in the opposite direction.

\begin{defn}
Given a pair $(\mathcal{S},\nabla)$ in $\Inv(Y/X)$ we define $\mathfrak{F}(\mathcal{S},\nabla)= \mathcal{S}^{\nabla}$. Moreover if we have two pairs $(\mathcal{S}_1,\nabla_1)$ and $(\mathcal{S}_2,\nabla_2)$ in $\Inv(Y/X)$ and if $f\colon \mathcal{S}_1 \to \mathcal{S}_2$ is a morphism in $\Inv(Y/X)$, we define $\mathfrak{F}(f)=f|_{(\mathcal{S}_1)^{\nabla_1}} \colon (\mathcal{S}_1)^{\nabla_1} \to (\mathcal{S}_2)^{\nabla_2}$.
\end{defn}

\begin{defn}\label{la definicion del functor G}
 Given a $\Ou_X$-coherent sheaf $\mathcal{S}_1$ we define $\mathfrak{G}(\mathcal{S}_1)=(\Ou_Y \otimes_{\Ou_X} \mathcal{S}_1 ,d\otimes_{\Ou_X}1_{\mathcal{S}_1})$. Moreover if $f_1:\mathcal{S}_1\to \mathcal{T}_1$ is a $\Ou_X$-morphism we define $\mathfrak{G}(f_1)=1_{\Ou_Y} \otimes_{\Ou_X} f_1$.
\end{defn}

\begin{lemm}\label{pullback belongs to inv}
    With the notations of the previous definition, we have that $\mathfrak{G}(f_1)$ belongs to the category $\Inv(Y/X)$ 
\end{lemm}
\begin{proof}
If $U$ is an affine open subset of $X$,  $V=\pi^{-1}(U)$ and $D\in \der_{\Ou_X(U)}\Ou_Y(V)$, then, over these open subsets, we have that 
$\mathfrak{G}(f_1)\circ(d\otimes_{\Ou_X}1_{\mathcal{S}_1})_D=D \otimes_{\Ou_X} f_1=(d\otimes_{\Ou_X}1_{\mathcal{T}_1})_D\circ\mathfrak{G}(f_1)$.
\end{proof}

From now on we will proceed to prove that these two functors establish an isomorphism between the above two categories, which gives the announced characterization in the beginning of the section. To this end, we will need the following important tool.

\begin{defn} \label{definicion de la tate trace}
   
Suppose that $V$ is an open subset of $Y$ such that $\Omega_{Y/X}(V)$ is free and take $U=\pi(V)$. We say that $s_1,\dots,s_r \in \Ou_Y(V)$ is a differential basis of $\Ou_Y(V)$ over $\Ou_X(U)$ when $ds_1,\ldots,ds_r$ is a $\Ou_Y(V)$-basis of $\Omega_{Y/X}(V)$. 

In this case, let  $D_1,\ldots,D_r\in \der_{\Ou_X(U)}\Ou_Y(V)$ be a dual basis of $ds_1,\ldots,ds_r$ (i.e. $D_i(s_j)=\delta_{ij}$). If $(\mathcal{S}, \nabla)$  is a pair in  $\Inv(Y/X)$,
then we define the 
trace $\tau_{s_1,\ldots,s_r}$ with respect to $s_1\ldots,s_r$ and $\nabla$
by  
$$\tau_{s_1,\ldots,s_r}=\nabla_{D_1}^{p-1}\circ\dots\circ\nabla_{D_r}^{p-1}\colon \mathcal{S}(V) \to \mathcal{S}(V).$$
\end{defn}

\begin{Remark}
We claim that if $\Omega_{Y/X}$ is locally free $\Ou_Y$-coherent sheaf, then there is a cover $\{ V_P \}_{P \in Y}$ of $Y$ by open subsets such that $\Ou_Y(V_P)$ has a differential basis over $\Ou_X(U_P)$, where $U_P=\pi(V_P)$.

Indeed, for each $P\in Y$, if $\m_{Y,P}$ is the the maximal ideal of the local ring $\Ou_{Y,P}$, then $(\Omega_{Y/X})_P/\m_{Y,P}(\Omega_{Y/X})_P$ is a $\textbf{k}(P)$-vector space of finite dimension, namely $r$, where $\textbf{k}(P)$ is the residue field at $P$. It is easy to see that, if $\sum_{i=1}^{m}y_{i,j}dx_i$ ($j=1,\ldots, r$) is basis of the above vector space, then we also may collect, from the generators $dx_1, \ldots, dx_m$, another basis of such vector space of the form $ds_1, \ldots, ds_r$. Now from Nakayama's Lemma we obtain that $(\Omega_{Y/X})_P = \Ou_{Y,P}ds_1 \oplus \cdots \oplus \Ou_{Y,P}ds_r$ and  we only need to consider the open neighborhood $V_P$ at $P$ such that $\Omega_{Y/X}(V_P)$ is free of rank $r$ and $s_1,\ldots,s_r\in \Ou_Y(V_P)$.
\end{Remark}

\begin{Remark}\label{propiedad de la traza}
Notice that, with the notation introduced in the above definition,   the pushforward  under $\pi$ of the image of  $\tau_{s_1,\ldots,s_r}$  is a sub-sheaf of $\mathcal{S}^{\nabla}$. To see this, from Proposition \ref{igualdad de conexiones} we only need to prove that   $\nabla_{D_i}\circ \tau_{s_1,\ldots,s_r}=0$, which in turn can be obtained if we prove that $\nabla_{D_i}\circ \nabla_{D_j} - \nabla_{D_j}\circ \nabla_{D_i}=\nabla_{[D_i,D_j]}=0$ and $\nabla_{D_i}^p=\nabla_{D_i^p}=0$. To show these two identities, it is enough to prove that $[D_i,D_j]=0=D_i^p$ for every $i$ and $j$ such that $i\neq j$. Indeed, they follow from the fact that for every derivation $D$, the identity $D=D_d\circ d$ implies that $D=0$ if and only if $D(s_i)=0$ for every $i$. 
\end{Remark}

\begin{lemm}\label{Igualdad de la imagen de nabla}
Let $V$ be an open subset of $Y$, $U=\pi(V)$, $s_1,\ldots,s_r$ be a differential basis of $\Ou_Y(V)$ over $\Ou_X(U)$, $(\mathcal{S},\nabla)$ be a pair in $\Inv(Y/X)$ and $\tau_{s_1,\ldots,s_r}$ be the trace with respect to $s_1,\ldots,s_r$.
If we denote by $s^{I}=s_1^{i_1}\ldots s_r^{i_r}$, for each $I=(i_1,\ldots,i_r) \in \{0,\ldots,p-1\}^r$, then 
\[
\tau_{s_1,\ldots,s_r}(s^Ix)= \left\{
\begin{array}{ccc}
    0 & \mbox{ for } & I\neq(p-1,\ldots,p-1)\\
    (-1)^rx & \mbox{ for } & I=(p-1,\ldots,p-1)
\end{array} \right.
\]
for all $x\in \mathcal{S}^{\nabla}(U)$. In particular, the image of $\mathcal{S}(V)$ by $\tau_{s_1,\ldots,s_r}$ is $\mathcal{S}^{\nabla}(U)$. 

\end{lemm}
\begin{proof}
To prove the first claim. To do this, we observe that if $D_1,\ldots,D_r\in \der_{\Ou_X(U)}\Ou_Y(V)$ is the dual basis of $ds_1,\ldots,ds_r$, then the identities
\begin{equation}\label{eq1}
    (D_i)^{p-1}(s_i^k)= \left\{
\begin{array}{ccc}
    0 & \mbox{ for } & k< p-1\\
    (p-1)! & \mbox{ for } & k=p-1
\end{array} \right.
\end{equation}

provide that
\begin{equation}\label{eq2}
    D_1^{p-1}(\ldots(D_r^{p-1}(s^I)\ldots)= \left\{
\begin{array}{ccc}
    0 & \mbox{ for } & I\neq (p-1,\ldots,p-1)\\
    (-1)^r & \mbox{ for } & I= (p-1,\ldots,p-1)
\end{array} \right. 
\end{equation}
which in turn, implies the stated identity.

From Remark \ref{propiedad de la traza}, we only need to show that $\mathcal{S}^{\nabla}(U)\subset \tau_{s_1,\ldots,s_r}(\mathcal{S}(V))$, in order to obtain the second claim. Indeed, if $x\in \mathcal{S}^{\nabla}(U)$ then it follows from (\ref{eq2}) that  $x=\tau_{s_1,\ldots,s_r}((-1)^rs^{(p-1,\ldots,p-1)}x)$. 
\end{proof}
\begin{prop}\label{Prefacio a eigualdad clave}\label{igualdad clave de haces invariantes}
 If we consider a pair $(\mathcal{S},\nabla)$  in $\Inv(Y/X)$ and a differential basis $s_1,\ldots,s_r$  of $\Ou_Y(V)$ over $\Ou_X(U)$, where $V$ is an open subset of $Y$ and $U=\pi(V)$, then \[\mathcal{S}(V) = \bigoplus_{I\in \{0,\ldots,p-1\}^r} s^I(\mathcal{S}^{\nabla}(U)).\]
 Moreover, every $x\in \mathcal{S}(V)$ can be written uniquely by \[x=(-1)^r\sum_{I\in \{0,\ldots,p-1\}^r} s^I \tau_{s_1,\ldots,s_r}(s^{(p-1,\ldots,p-1)-I}x).\]
\end{prop}
\begin{proof}
For each $i=1,\ldots,r$ let us fix $\mathcal{S}_i:=\{x\in \mathcal{S}(V)\; ; \;\nabla_{D_j}(x)=0 \mbox{ for }j\geq i\}$ and $\mathcal{T}_{i}:= \sum_{k=0}^{p-1} s_i^k\mathcal{S}_i \subseteq \mathcal{S}_{i+1}$, where $D_1,\ldots,D_r\in \der_{\Ou_X(U)}\Ou_Y(V)$ is the dual basis of $ds_1,\ldots,ds_r$, and we define $S_{r+1}:=S(V)$. By using Lemma \ref{igualdad de conexiones}, we have that $\mathcal{S}_1=\mathcal{S}^{\nabla}(U)$ and, therefore, we only need to prove that $\mathcal{T}_{i}=\mathcal{S}_{i+1}$ for $i=1,\ldots,r$. If we take $x\in \mathcal{S}_{i+1}$, then (\ref{eq1}) says that $D_i^p=0$ and, hence, $\nabla_{D_i}^p(x)=\nabla_{D_i^p}(x)=0\in \mathcal{T}_i$. Let us suppose that the minimum $j$ such that $(\nabla_{D_i})^j(x)=y\in  \mathcal{T}_i$ is bigger than zero. From the definition of $\mathcal{T}_i$ we may write 
$y=\sum_{k=0}^{p-1} s_i^k y_k$ with $y_k \in \mathcal{S}_i$. By applying Lemma  \ref{Calculo de conexxiones} recursively we have
$\nabla^{p-1}_{D_i}(s_i^ky_k)=\nabla^{p-2}_{D_i}(D_i(s_i^k)y_k)=\ldots=D_i^{p-1}(s_i^{k})y_k$, 
which implies that $\nabla^{p-1}_{D_i}(s_i^ky_k)=0$ for $k\neq p-1$, $\nabla^{p-1}_{D_i}(s_i^{p-1}y_{p-1})=(p-1)!y_{p-1}$ and so\[\frac{\nabla_{D_i}^{p-1}(y)}{(p-1)!}=y_{p-1}.\]
Using the assumption $j>0$ we have $\nabla_{D_i}^{p-1}(y)=\nabla_{D_i}^{p-1+j}(x)=0$, that is, $y_{p-1}=0$ and $y=\sum_{k=0}^{p-2} s_i^k y_k$. If we consider \[y'=-\sum_{k=0}^{p-2} \frac{s_i^{k+1}}{k+1} y_{k}\]
then $\nabla_{D_i}(y')=-y$ and  $\nabla_{D_i}(\nabla_{D_i}^{j-1}(x)+y')=\nabla_{D_i}^j(x)-y=0$.
 On the other hand $\nabla_{D_l}(\nabla_{D_i}^{j-1}(x)+y')=0$ for all $l\geq i$ implies that $\nabla_{D_i}^{j-1}(x)+y'\in \mathcal{S}_i \subset \mathcal{T}_i$. In this way, we conclude that $\nabla_{D_i}^{j-1}(x)\in \mathcal{T}_i$,  contradicting the assumption on $j$. So $j=0$ and $x\in \mathcal{T}_i$. Since  $\mathcal{S}_i\subset \mathcal{S}_{i+1}$ we can use Lemma \ref{Calculo de conexxiones} to conclude that $\mathcal{T}_i= \mathcal{S}_{i+1}$.

 Now, from the identity that we have obtained above, any $x \in S(V)$ can be written by   $x=\sum_{I\in \{0,\ldots,p-1\}^r} s^I x_I$ where $x_I\in \mathcal{S}^{\nabla}(U)$. Therefore, to finish the proof, we only need to observe that  Lemma \ref{Igualdad de la imagen de nabla} implies:  $(-1)^r \tau_{s_1,\ldots,s_r}(s^{(p-1,\ldots,p-1)-I}x)=x_I$.
\end{proof}

\begin{prop}\label{igualdad espacios}
If $(\mathcal{S},\nabla)$ is a  pair in the category $\Inv(Y/X)$, then the natural $\Ou_Y$-morphism $\phi: \Ou_Y\otimes_{\Ou_X}\mathcal{S}^{\nabla}\to \mathcal{S}$ is an isomorphism.
\end{prop}
\begin{proof}
Suppose that $s_1,\ldots, s_r$ is a differential basis of $\Ou_Y(V)$ over $\Ou_X(U)$, where $V$ is an open subset of $Y$ and $U=\pi(V)$. Then we can define $\psi:\mathcal{S}(V)\to \Ou_Y(V) \otimes\mathcal{S}^{\nabla}(U) $ by $$\psi(x)=(-1)^r\sum_{I\in \{0,\ldots,p-1\}^r} s^I\otimes \tau_{s_1,\ldots,s_r}(s^{(p-1,\ldots,p-1)-I}x) .$$
Clearly Proposition \ref{igualdad clave de haces invariantes} says  that $\phi\circ \psi=1_{\mathcal{S}}$. On the other hand, by using  Lemma \ref{Igualdad de la imagen de nabla} we conclude that $\psi\circ \phi(s^J\otimes  x)= s^J \otimes x $ for every $x\in \mathcal{S}^{\nabla}(U)$  and every $J\in \{0,\ldots,p-1\}^r$. Therefore, $\psi\circ \phi=1_{ \Ou_Y \otimes \mathcal{S}^{\nabla}}(V)$.
\end{proof}

The next two theorems will obtain the claimed isomorphism between objects from the categories $\Inv(Y/X)$ and ``$\Ou_X$-coherent sheaves", which is an alternative to Cartier's theorem  (cf. Theorem 5.1, \cite{Katz}) for purely inseparable morphisms. 

\begin{thm}\label{Funtor igualdad de ida}
If  $(\mathcal{S},\nabla)$ is a  pair in the category $\Inv(Y/X)$,   then $\mathfrak{G}(\mathfrak{F}(\mathcal{S},\nabla))\simeq (\mathcal{S},\nabla)$.
\end{thm}
\begin{proof}
From the definition of the functor $\mathfrak{G}$ and from Proposition \ref{igualdad espacios} we only need to prove that $\nabla_2:=d\otimes_{\Ou_X} 1_{\mathcal{S}^{\nabla}}$ is equal to $\nabla$, where such identity must take into account the isomorphism from the above proposition. As in the previous proofs, our assumption is that $\Ou_Y(V)$ admits a differential basis over $\Ou_X(U)$. Notice that, given $x\in S(V)$, Proposition \ref{Prefacio a eigualdad clave} provides the written $x=\sum a_ix_i$, with $x_i\in \mathcal{S}^{\nabla}(U)$ and $a_i\in \Ou_Y(V)$.  For each $D\in \der_{\Ou_X(U)}(\Ou_Y(V))$,  we have $\nabla_{D}(x)=\nabla_{D}\left(\sum a_ix_i\right)=\sum D(a_i)x_i=(\nabla_2)_{D}\left(\sum a_ix_i\right)=(\nabla_2)_{D}(x)$, from Lemma \ref{Calculo de conexxiones}, because $\nabla_2(x_i)= d(1)\otimes x_i=0=\nabla(x_i)$. Therefore using Lemma \ref{igualdad de conexiones} for the $\Ou_X$-morphism $\nabla -\nabla_2$ we conclude that $\nabla =\nabla_2$.
\end{proof}

\begin{thm}\label{funtor igualdad de vuelta}
 If   $\mathcal{S}_1$ is a $\Ou_X$-coherent sheaf, then  $\mathfrak{F}(\mathfrak{G}(\mathcal{S}_1))\simeq \mathcal{S}_1$.
\end{thm}
\begin{proof}
From definition, $\mathfrak{G}(\mathcal{S}_1)=(\Ou_Y \otimes_{\Ou_X} \mathcal{S}_1 ,d\otimes_{\Ou_X}1_{\mathcal{S}_1})$ and  $\nabla = d\otimes_{\Ou_X}1_{\mathcal{S}_1}$.  
Let $\gamma: \mathcal{S}_1\to \Ou_Y\otimes_{\Ou_X}\mathcal{S}_1$ be the natural morphism given by $\gamma(x)=1\otimes x$. Since $\nabla\circ\gamma=0$ we can induce a $\Ou_X$-morphism $\gamma:\mathcal{S}_1\to (\Ou_Y\otimes_{\Ou_X}\mathcal{S}_1)^{\nabla}$.
As we can find a cover $\{ V_i \}$ of $Y$ by affine open subsets such that every $\Ou_Y(V_i)$ admit a differential basis over $\Ou_X(\pi(V_i))$, then it is enough to show that $\gamma\mid_{ \pi(V)}$ is an isomorphism whenever $\Ou_Y(V)$ admits a differential basis over $\Ou_X(\pi(V))$. Such claim follows from the three properties/facts: firstly we will show that the inverse of $\gamma\mid_{ \pi(V)}$ will be defined by using a differential basis of $\Ou_Y(V)$; secondly we notice that every differential basis of $\Ou_Y(V)$ over $\Ou_X(\pi(V))$ can be restricted to a differential basis for each open subset of $V$; the inverse of $\gamma\mid_{ \pi(V)}$ is uniquely determined. Indeed, with these three properties we obtain that the local inverses can be glued to produce a global inverse for $\gamma$. 

To show the first property observed above, we take an element $y\in (\Ou_Y\otimes_{\Ou_X}\mathcal{S}_1)^{\nabla}(\pi(V))$ and we define $\beta(y):=(-1)^r\tau_{s_1,\ldots,s_r}(s^{(p-1,\ldots,p-1)}y) \in (\Ou_Y\otimes_{\Ou_X}\mathcal{S}_1)(V)$, where $s_1,\ldots,s_r$ is a differential basis of $\Ou_Y(V)$ over $\Ou_X(\pi(V))$. From Lemma \ref{Igualdad de la imagen de nabla}, for each $x \in \mathcal{S}_1(\pi(V))$, we have   $\beta\circ\gamma(x)= \beta(1\otimes x) = (-1)^r\tau_{s_1,\ldots,s_r}(s^{(p-1,\ldots,p-1)}(1\otimes x))) = 1\otimes x \in (\Ou_X \otimes_{\Ou_X} \mathcal{S}_1)(\pi(V))$. Therefore  $(\eta \circ \beta) \circ \gamma\mid_{ \pi(V)} = 1_{\mathcal{S}_1}$, where $\eta \colon \Ou_X \otimes_{\Ou_X} \mathcal{S}_1 \to \mathcal{S}_1$ is defined by $\eta(a\otimes x)=ax$.

If $x\in (\Ou_Y\otimes_{\Ou_X}\mathcal{S}_1)^{\nabla}( \pi(V))$ we may write $x=\sum   s^I \otimes x_I$ where $x_I\in\mathcal{S}_1( \pi(V))$ and, from
 Lemma \ref{Igualdad de la imagen de nabla}, we obtain that $\tau_{s_1,\ldots,s_r}(s^{(p-1,\ldots,p-1)-J}x)=\tau_{s_1,\ldots,s_r}(s^{(p-1,\ldots,p-1)}(1\otimes x_J)) = (-1)^r (1\otimes x_J). $
Moreover, the same lemma says that $\tau_{s_1,\ldots,s_r}(s^{(p-1,\ldots,p-1)-J}x)=0$ for $J\neq (0,\ldots,0)$. Then $x=1\otimes x_{(0,\ldots,0)}$ and $\eta\circ\beta(x)=x_{(0,\ldots,0)}$ and, therefore, $\eta\circ\beta$ is the inverse of $\gamma\mid_{ \pi(V)}$.
\end{proof}

\begin{Remark}
    It is also possible to prove that the functors $\mathfrak{F}$ and $\mathfrak{G}$ provide an equivalence between the categories $\Inv(Y/X)$ and the category of $\Ou_X$-coherent sheaves, but we decided to omit the proof because it is not useful for the rest of this paper.
\end{Remark}

\subsection{$X$-invariant ideal sheaves and $X$-invariant subvarieties of $Y$} \label{invariant subvarieties of Y}

In this subsection, we will take advantage of the isomorphism of functors in Subsection \ref{section invariant sheaves} to characterize ideal sheaves of $Y$ (respectively, subvarieties of $Y$) obtained from pulling-back ideal sheaves of $X$ (respectively, subvarieties of $X$). In this way, we will keep the notations introduced up to now.

\begin{defn}\label{def invarant ideal}
 Given an ideal sheaf $\mathcal{I}$ of $Y$, let $\nabla$ be the restriction to $\mathcal{I}$ of the connection $d$ in $\Ou_Y$, constructed in Proposition \ref{d integrable}, in which $(\Ou_Y,d) \in \Inv(Y/X)$. 
 We say that the ideal sheaf $\mathcal{I}$ is $X$-invariant if 
 $(\mathcal{I},\nabla)$ and $i:\mathcal{I}\to \Ou_Y$ belong to  $\Inv(Y/X)$, were  $i$ is the natural inclusion.
\end{defn}

Next proposition says that $X$-invariant ideal sheaves are precisely the ones that we have been looking for.

\begin{prop}\label{Ideais Invarantes}
Let $\mathcal{I}$ be an ideal sheaf of $Y$. The following conditions are equivalent.
\begin{enumerate}
    \item The sheaf $\mathcal{I}$ is $X$-invariant;
    \item The sheaf $\mathcal{I}$ is equal to the base change $\mathcal{J}\otimes_{\Ou_X}\Ou_Y$ of an ideal sheaf $\mathcal{J}$ of $X$;
    \item For every affine open subset $U$ of $X$, for $V=\pi_1^{-1}(U)$ and for every derivation $D\in \der_{\Ou_X(U)}(\Ou_Y(V))$ we have  $D(\mathcal{I}(V))\subset \mathcal{I}(V)$;
    \item If $d$ is the natural connection of $\Ou_Y$, as in Proposition \ref{d integrable}, then $d(\mathcal{I})\subset \mathcal{I}\otimes_{\Ou_Y} \Omega_{Y/X}$. 
\end{enumerate}
\end{prop}
\begin{proof}
($1\then 2$) Follows from Theorem \ref{Funtor igualdad de ida}. ($2\then 3$) If $\mathcal{I}=\mathcal{J}\otimes_{\Ou_X}\Ou_Y$ were $\mathcal{J}$ is an ideal sheaf on $\Ou_X$, then for each $x\in \mathcal{I}(V)$, we can write $x=\sum a_ix_i$ with $x_i\in \mathcal{J}(U)\subset \Ou_X(U)$ and $a_i \in \Ou_Y(V)$. Therefore $D(x)=\sum D(a_i)x_i\in\mathcal{I}(V)$ for every $D\in \der_{\Ou_X(U)}\Ou_Y(V)$. ($3\then 4$) 
Suppose that $\Omega_{Y/X}(V)$ is free. Let $s_1,\dots,s_r$ be a differential basis of $\Ou_Y(V)$ over $\Ou_X(U)$ and  $D_i\in \der_{\Ou_X(U)}\Ou_Y(V)$ such that $D_i(s_j)=\delta_{ij}$. Given $x\in \mathcal{I}(V)$ we will show that $dx\in (\mathcal{I}\otimes_{\Ou_Y} \Omega_{Y/X})(V)$. Using the proof of Lemma \ref{igualdad de conexiones} (without assuming the hypothesis there) we can see that $dx=\sum d_{D_i}(x)\otimes ds_i$, where $d_{D_i}$ is constructed in Definition \ref{definition of connections}. From hypothesis  $d_{D_i}(x)\in \mathcal{I}(V)$, which implies that  $dx\in (\mathcal{I}\otimes_{\Ou_Y}\Omega_{Y/X})(V)$. ($4\then 1$) In this case the hypothesis says that $d_{\mid_{\mathcal{I}}}$ is a connection from $\mathcal{I}$ to $\mathcal{I}\otimes_{\Ou_X}\Omega_{Y/X}$ and its integrability and zero $p$-curvature are inherited from that of $d$. Clearly $i$ belongs to $\Inv(Y/X)$.
\end{proof}

From Definition \ref{def invarant ideal} and Proposition \ref{Ideais Invarantes} we are able to define the objects that we have promised to study since the beginning of this chapter.

\begin{defn}
We say that a subvariety $Z$ of $Y$ is $X$-invariant  if $\mathcal{I}_Z$ is $X$-invariant.
\end{defn}

\begin{prop}\label{i categoria Mic}
Given an $X$-invariant subvariety $Z$ of $Y$ with  the natural inclusion $i:Z\to Y$, then there is a subvariety $Z_1$ of $X$ such that $Z=Z_1\times_X Y \simeq \pi^* (Z_1)$. Moreover there is a connection $\nabla$ on $\Ou_Z$ such that $(\Ou_Z,\nabla)$ and  $i^{\ast}\colon \Ou_Y\to \Ou_{Z}$ belong to the category $\Inv(Y/X)$.
\end{prop}
\begin{proof}
We will prove that there exists a connection $\nabla$ on $\Ou_Z$ such that $(\Ou_Z,\nabla)$ and  $i^{\ast}:\Ou_Y\to \Ou_Z$ belong to  $\Inv(Y/X)$.
Since $(\mathcal{I}_Z,d_{\mid_{\mathcal{I}_Z}})$ belongs to $\Inv(Y/X)$ then there is an ideal sheaf  $\mathcal{I}_{Z_1}$ of $\Ou_X$ such that $\mathcal{I}_Z=\mathcal{I}_{Z_1}\otimes_{\Ou_X}\Ou_Y$. By tensoring the exact sequence $0\to \mathcal{I}_{Z_1}\to \Ou_X\to \Ou_{Z_1}\to 0$ with $\Ou_Y$ we obtain the exact sequence $ 0 \to \mathcal{I}_{Z}\to \Ou_Y\to \Ou_{Z_1}\otimes_{\Ou_X}\Ou_Y\to 0$, since $\Omega_{Y/X}$ locally free implies the flatness of $\Ou_Y$ from Proposition \ref{Prefacio a eigualdad clave}.  Therefore $\Ou_{Z_1}\otimes_{\Ou_X}\Ou_{Y}$ is isomorphic to $\Ou_Z$ (so $Z= Z_1\times_X Y$) and, in particular, there is an integrable connection $\nabla$ over $\Ou_Z$  such that  $(\Ou_Z,\nabla)$ belongs to the category $\Inv(Y/X)$. Notice that $ i^{\ast}:\Ou_{Y}\to \Ou_Z$ can be written by $i^{\ast}=i_1^{\ast}\otimes_{\Ou_X} 1_{\Ou_Y}$ where $i_1^{\ast}:\Ou_X\to \Ou_{Z_1}$ is the natural projection. Hence $i^{\ast}$ belong to  $\Inv(Y/X)$ from Proposition \ref{morphism in MIC}.
\end{proof}

We conclude this subsection by reformulating our results in the context of the foliation setting. As we can see in Proposition 1.9 of \cite{Miy}, there is a foliation $\mathcal{F}=\der_{\Ou_X} \Ou_Y$ such that $X=Y^{\mathcal{F}}$, that is, $\Ou_X=\operatorname{Ann}(\mathcal{F})=\{a \in \Ou_Y\,;\, D (a)=0 \text{ for every } D \in \mathcal{F}\}$. A subvariety $Z$ of $Y$ is said to be invariant by $\mathcal{F}$ when $D(\mathcal{I}_Z) \subset \mathcal{I}_Z$, for every $D \in \mathcal{F}$, where $\mathcal{I}_Z$ is the ideal sheaf of $Z$. Therefore, Proposition \ref{Ideais Invarantes} says that $X$-invariant subvarieties of $Y$ correspond exactly to subvarieties of $Y$ which are invariant by $\mathcal{F}$.

\section{Local properties at non-smooth regular points of geometrically integral curves} \label{local invariants at non-smooth points}

Let $Z$ be a geometrically integral scheme over $K$ and $P$ be a point of $Z$, where $K$ is a non-perfect and separably closed field of characteristic $p>0$. We denote the local ring of $Z$ at $P$ by $\Ou_{Z,P}$, its maximal ideal by $\m_{Z,P}$ and its residue field by $\textbf{k}(Z,P)$.

Suppose that $L$ is an algebraic extension of $K$. Since $Z$ is geometrically integral we notice that $\Ou_{Z,P}$ is a subring of the domain $\Ou_{Z^L,P^L}$, where $Z^L:=Z\times_{\spec K}\spec L$ and $P^L$ is the unique point of $Z^L$ lying over $P$. If $x\in \Ou_{Z,P}$ and $\alpha\in L$ we can  identify $\alpha x$ with $\alpha\otimes x$. With such identification, $\Ou_{Z^L,P^L}$ is given by $\Ou_{Z,P}L$. It is natural to expect non-regularity at $P^L$ and, in this case, we will denote the point of the normalization $\widetilde{Z^L}$ of $Z^L$, lying over it, by $\widetilde{P^L}$.

In this section we will study numerical properties at non-smooth points of a curve $C$, over $K$, and at the singularities of $C^{\overline{K}}$, lying over them.

\subsection{A criteria for regularity of geometrically integral curves} \label{subsection 3.1}

Let $C$ be a geometrically integral curve over $K$ and $P$ be a non-smooth point of $C$. In this subsection we will determine conditions to guarantee the regularity at $P$.

\begin{lemm}\label{generadores de los m_i}
The maximal ideal $\m_{{Z^{\overline{K}},P^{\overline{K}}}}$ is generated by elements of the form $x-a_x$ where $x\in \Ou_{Z,P}$ and $a_x\in \textbf{k}(Z,P)$.
\end{lemm}
\begin{proof}
Since $\Ou_{Z,P}/\m_{Z,P}\subset \overline{K}$ then for every $x\in \Ou_{Z,P}$ there exists $a_x\in \textbf{k}(Z,P)$ such that  $x-a_x\in \m_{Z^{\overline{K}},P^{\overline{K}}}$. On the other hand, if $y\in \m_{{Z^{\overline{K}},P^{\overline{K}}}}$ then $y=\sum b_jx_j$ where $b_j\in  \overline{K}$ and $x_j\in \Ou_{Z,P}$. If we write $y=\sum b_j(x_j-a_{x_j})+\sum b_ja_{x_j}$ where $x_j-a_j\in\m_{{Z^{\overline{K}},P^{\overline{K}}}}$ we obtain $\sum b_ja_{x_j}\in \overline{K}\cap \m_{{Z^{\overline{K}},P^{\overline{K}}}}$, that is,  $\sum b_ja_{x_j}=0$ and $y=\sum b_j(x_j-a_{x_j}).$
\end{proof}


\begin{lemm}\label{equaliti of e imly2}
If $L'|L$ is an algebraic extension such that $L'\otimes_L\overline{K}$ is a field then $L'=L$.
\end{lemm}
\begin{proof}
 Suppose that $L\neq L'$. Hence there exists $x\in L'$ such that $x\not \in L$. Since $K$ is separably closed then $L'|L$ is purely inseparable and we can take the smallest positive integer $n$ such that $x^{p^n}=y\in L$, that is, $L[x]= L[T]/\langle T^{p^n}-y\rangle$ is a field between $L$ and  $L'$. Since $L[x]\otimes_L\overline{K}= \overline{K}[T]/\langle T^{p^n}-y\rangle$ is a subring of the field $L'\otimes_L\overline{K}$, then it is reduced, contradicting the nilpotence of the element $\overline{T-x}$.  \end{proof}

\begin{prop}\label{pre grado emb}
If the classes of $x_1-a_1,\ldots,x_l-a_l$ form a basis for the $\overline{K}$-vector space $$\frac{\m_{Z^{\overline{K}},P^{\overline{K}}}}{\m_{Z^{\overline{K}},P^{\overline{K}}}^2+\overline{K}\m_{{Z,P}}}$$ with $x_i\in \Ou_{Z,P}$ and $a_i \in \textbf{k}(Z,P)$, then $\textbf{k}(Z,P)=K(a_1,\ldots,a_l)$. Moreover $da_1,\ldots,da_l$ is a basis of $\Omega_{\textbf{k}(Z,P)/K}$.
\end{prop}
\begin{proof}
Firstly we prove that if $L= K(a_1,\ldots,a_l)$, then the inclusion of fields $L\subseteq \textbf{k}(Z,P) \subseteq \textbf{k}(Z^L,P^L)$ become an identity. Indeed, from hypothesis, the local ring $\Ou_{Z^{\overline{K}},P^{\overline{K}}}=\Ou_{Z^L,P^L}\otimes_L\overline{K}$ has maximal ideal $\m_{Z^{\overline{K}},P^{\overline{K}}}=\m_{Z^L,P^L}\otimes_L\overline{K}$  and residue field $\textbf{k}(Z^{\overline{K}},P^{\overline{K}})=\textbf{k}(Z^L,P^L)\otimes_L \overline{K}$. By applying Lemma \ref{equaliti of e imly2} we obtain $\textbf{k}(Z^L,P^L)=L$.

Before proving the second claim, we notice that $\textbf{k}(Z,P)=K(a_1,\ldots,a_l)$ implies that
the $K$-vector space $\Omega_{\textbf{k}(Z,P)/K}$ is spanned by $da_1,\ldots,da_{l}$ and thus, after reordering $a_1,\ldots,a_l$, we can assume that  
$da_1,\ldots,da_{l'}$ is a basis of $\Omega_{\textbf{k}(Z,P)/K}$ with $l'\leq l$. Now we will prove that $l=l'$. From Proposition 26.5 in \cite{Mats}, $a_1,\ldots,a_{l'}$ is a $p$-basis for $\textbf{k}(Z,P)$. Hence $\textbf{k}(Z,P)= K(a_1,\ldots,a_{l'})$ provides that for any element in $y\in \Ou_{Z,P}$, its residue class $a_y$ in $\textbf{k}(Z,P)$ is given by $a_y=F(a_1,\ldots,a_{l'})$ where $F\in K[X_1,\ldots,X_{l'}]$. In particular $z:=y-F(x_1,\ldots,x_{l'}) \in \m_{Z,P}$.  From the identity in $\Ou_{Z^{\overline{K}},P^{\overline{K}}}$,  $F(x_1,\ldots, x_{l'})=F(a_1,\ldots,a_{l'})-\sum_{i=1}^{l'} \frac{\partial F}{\partial X_i}(a_1,\ldots,a_{l'})(x_i-a_i)+r$, with $r\in \m_{Z^{\overline{K}},P^{\overline{K}}}^2$, we conclude that 
$y-a_y= -\sum_{i=1}^{l'} \frac{\partial F}{\partial X_i}(a_1,\ldots,a_{l'})(x_i-a_i)+r+z.$ Therefore, Lemma \ref{generadores de los m_i} says that $x_1-a_1,\ldots,x_{l'}-a_{l'}$ generate the space $\m_{Z^{\overline{K}},P^{\overline{K}}} / (\m_{Z^{\overline{K}},P^{\overline{K}}}^2+\overline{K}\m_{{Z,P}})$,
implying, from hypothesis, that $l=l'$.\end{proof}
\begin{prop}\label{basis canonica de embebimento}
Let $C$ be a geometrically integral curve over $K$ and $P \in C$ be a point lying under a singularity $P^{\overline{K}}$ of $C^{\overline{K}}$. If  $C$ is regular at $P$, then
 \[
\dim_{\textbf{k}(C,P)} \Omega_{\textbf{k}(C,P)/K}\leq \dim_{\overline{K}}\frac{\m_{C^{\overline{K}},P^{\overline{K}}}}{\m_{C^{\overline{K}},P^{\overline{K}}}^2}\leq  \dim_{\textbf{k}(C,P)} \Omega_{\textbf{k}(C,P)/K}+1. 
  \]
\end{prop}
\begin{proof}
 Let $x_1-a_1,\ldots,x_l-a_l$ be a basis of  $\m_{C^{\overline{K}},P^{\overline{K}}}/(\m_{C^{\overline{K}},P^{\overline{K}}}^2+\overline{K}\m_{C,P})$ as in the above proposition. Therefore the exact sequence 
 \[
 0\to \frac{\overline{K}\m_{C,P}+\m_{C^{\overline{K}},P^{\overline{K}}}^2}{\m_{C^{\overline{K}},P^{\overline{K}}}^2} \to \frac{
\m_{C^{\overline{K}},P^{\overline{K}}}}{\m_{C^{\overline{K}},P^{\overline{K}}}^2} \to \frac{\m_{C^{\overline{K}},P^{\overline{K}}}}{\m_{C^{\overline{K}},P^{\overline{K}}}^2+\overline{K}\m_{C,P}} \to 0
\]
and the last proposition implies that $\dim_{\overline{K}}\frac{\m_{C^{\overline{K}},P^{\overline{K}}}}{\m_{C^{\overline{K}},P^{\overline{K}}}^2}=l+\dim_{\overline{K}}\frac{\overline{K}\m_{C,P}+\m_{C^{\overline{K}},P^{\overline{K}}}^2}{\m_{C^{\overline{K}},P^{\overline{K}}}^2}.$
On the other hand, from one of the isomorphism theorems 
$$\dim_{\overline{K}}\frac{\overline{K}\m_{C,P}+\m_{C^{\overline{K}},P^{\overline{K}}}^2}{\m_{C^{\overline{K}},P^{\overline{K}}}^2}=\dim_{\overline{K}}\frac{\overline{K}\m_{C,P}}{\m_{C^{\overline{K}},P^{\overline{K}}}^2\cap \overline{K}\m_{C,P}}\leq\dim_{\overline{K}}\frac{\overline{K}\m_{C,P}}{\overline{K}\m_{C,P}^2}=1$$ because $C$ is regular at $P$, which concludes the proof. 
\end{proof}
\begin{defn}
Let $C$ be a geometrically integral curve over $K$ and $P \in C$. Suppose that $\Ou$ is a subring of $\Ou_{\widetilde{C^{\overline{K}}},\widetilde{P^{\overline{K}}}}$ (where $\widetilde{C^{\overline{K}}}$ is the normalization of $C^{\overline{K}}$ and $\widetilde{P^{\overline{K}}}$ is the point of $\widetilde{C^{\overline{K}}}$ lying over $P$).  If $\nu_0$ is the discrete valuation of $K(C)\overline{K}|\overline{K}$ whose valuation ring is $\Ou_{\widetilde{C^{\overline{K}}},\widetilde{P^{\overline{K}}}}$, then we define  the semigroup of $\Ou$ by \[\Gamma_{\Ou}=\{l\in \mathbb{Z}\,|\, l=\nu_0(x) \mbox{ for some } x\in \Ou\}.\] We denote $\Gamma_{\Ou_{C,P}}$ by $\Gamma_{C,P}$  and $\Gamma_{\Ou_{C^{\overline{K}},P^{\overline{K}}}}$ by $\Gamma_{C^{\overline{K}},P^{\overline{K}}}$.
\end{defn}

\begin{prop}\label{semigrupo de un punto regular como se ve}
Let $C$ be a geometrically integral curve over $K$  and $P\in C$ be a point lying under a singularity $P^{\overline{K}}$ of $C^{\overline{K}}$.  If $[\textbf{k}(C,P):K]$ belongs to $\Gamma_{C,P}$ then $C$ is regular at $P$.
\end{prop}
\begin{proof}
    Let $t\in \Ou_{\widetilde{C},\widetilde{P}}$ be a uniformizing parameter at $\widetilde{P}$ (were $\widetilde{C}$ is the normalization of $C$ and $\widetilde{P}$ is the point of $\widetilde{C}$ lying over $P$). We recall that the proof of Lemma 1.5 in \cite{Bedoya} says that $[\textbf{k}(\widetilde{C},\widetilde{P}):K]=\nu_0(t)$. Since $[\textbf{k}(C,P):K]$ belongs to $\Gamma_{C,P}$ then there exist $y\in \Ou_{C,P}$ such that $\nu_0(y)=[\textbf{k}(C,P):K]$. But $y$ is a multiple of a power of $t$ in $\Ou_{\widetilde{C},\widetilde{P}}$ and, hence,
    $[\textbf{k}(C,P):K]$ $\leq [\textbf{k}(\widetilde{C},\widetilde{P}):K]= \nu_0(t)\leq \nu_0(y)$.     
Therefore $\textbf{k}(C,P)=\textbf{k}(\widetilde{C},\widetilde{P})$, $y$ is an uniformization parameter of $\Ou_{\widetilde{C},\widetilde{P}}$ and
$\Ou_{\widetilde{C},\widetilde{P}}=\Ou_{C,P}+y\Ou_{\widetilde{C},\widetilde{P}}$. Thus from Nayakama's Lemma, applied to the finitely generated $\Ou_{C,P}$-module $\Ou_{\widetilde{C},\widetilde{P}}$ (because $\dim_K \Ou_{\widetilde{C},\widetilde{P}}/\Ou_{C,P} < \infty$ from Theorem 1 of \cite{Ros}), we obtain that $\Ou_{\widetilde{C},\widetilde{P}}=\Ou_{C,P}$.    
\end{proof}

\begin{cor}\label{Criterio para la regularidad de puntos}
Let $C$ be a geometrically integral curve over $K$ and $P \in C$ be a point lying under a singularity $P^{\overline{K}}$ of $C^{\overline{K}}$. If there exists an invertible ideal sheaf $\mathcal{I}$ of $\Ou_C$  such that 
\[\dim_{\overline{K}}\frac{\Ou_{C^{\overline{K}},P^{\overline{K}}}}{\Ou_{C^{\overline{K}},P^{\overline{K}}}\mathcal{I}_{P}}=[\textbf{k}(C,P):K]\]
then $C$ is a regular at $P$.
\end{cor}
\begin{proof}
    We may assume that $\mathcal{I}_{P}$ is generated by $y \in \Ou_{C,P}$ and, then, we have
    \[
    \begin{array}{lll}
    \dim_{\overline{K}}\frac{\Ou_{\widetilde{C^{\overline{K}}},\widetilde{P^{\overline{K}}}}}{y\Ou_{\widetilde{C^{\overline{K}}},\widetilde{P^{\overline{K}}}}}    & =& \dim_{\overline{K}}\frac{\Ou_{\widetilde{C^{\overline{K}}},\widetilde{P^{\overline{K}}}}}{y\Ou_{C^{\overline{K}},P^{\overline{K}}}}- \dim_{\overline{K}}\frac{y\Ou_{\widetilde{C^{\overline{K}}},\widetilde{P^{\overline{K}}}}}{y\Ou_{C^{\overline{K}},P^{\overline{K}}}}   \\
        & =&\dim_{\overline{K}}\frac{\Ou_{\widetilde{C^{\overline{K}}},\widetilde{P^{\overline{K}}}}}{\Ou_{C^{\overline{K}},P^{\overline{K}}}}+\dim_{\overline{K}}\frac{\Ou_{C^{\overline{K}},P^{\overline{K}}}}{y\Ou_{C^{\overline{K}},P^{\overline{K}}}}- \dim_{\overline{K}}\frac{\Ou_{\widetilde{C^{\overline{K}}},\widetilde{P^{\overline{K}}}}}{\Ou_{C^{\overline{K}},P^{\overline{K}}}} \\
        &=&[\textbf{k}(C,P):K],
   \end{array}\]
    implying that $\nu_{0}(y)=[\textbf{k}(C,P):K]$, i.e.,  $[\textbf{k}(C,P):K]\in \Gamma_{C,P}$. Hence, the result follows from Proposition \ref{semigrupo de un punto regular como se ve}.
\end{proof}

\subsection{Semigroups at singularities lying over non-smooth regular points} \label{section 3.2}
In this subsection we keep using the notation of the last subsection and we will compute some invariants associated to non-smooth points (and their base extensions) of a geometrically integral and regular curve $C$ over $K$. To this end, we will need to use the relative Frobenius morphism $F_{C/K}:C\to C^{(p)}$ of $C$ whose definition can be found in \cite{Liu} page 94. 
For each $k\in \mathbb{N}$, we denote by $F_{C/K}^k$ the $k^{\mbox{th}}$-iterate of $F_{C/K}$. The image $C^{(p^k)}$ of $C$ under $F_{C/K}^k$ is a singular curve over $K$ and we need to study its normalization $C_k=\widetilde{C^{(p^k)}}$ to obtain properties of $C$. In this way we denote the point of $C_k$ lying over $F_{C/K}^k(P)$ by $P_k$. Finally we denote the discrete valuation of $\overline{K}K(C_k)|\overline{K}$ associated to $\Ou_{\widetilde{C_k^{\overline{K}}},\widetilde{P_k^{\overline{K}}}}$ by $\nu_k$.

Notice that $[K(C):K(C_1)]=p$ (see \cite{GoldSchmith}, Theorem 2.4.6, pg. 55) and notice also that there exists $x_0\in \Ou_{C,P}$ such that $\Ou_{C_1,P_1}[x_0]=\Ou_{C,P}$ (from Theorems 1.1.23 and 1.1.24 in \cite{GoldSchmith}). 

\begin{defn}
  Let us consider the map $\alpha:\Ou_{\widetilde{C^{\overline{K}}},\widetilde{P^{\overline{K}}}}\Omega_{ \overline{K}\Ou_{C,P} /\overline{K}}\to  \Omega_{\Ou_{\widetilde{C^{\overline{K}}},\widetilde{P^{\overline{K}}}}/\overline{K}}$
 where $\Omega_{\Ou/\overline{K}}$ is the module of K\"ahler differentials of the $\overline{K}$-algebra $\Ou$ over $\overline{K}$. The number  \[d(C,P)=1+\dim_{\overline{K}} \coker (\alpha)\]
will be called by the differential degree at $P$.
\end{defn}

\begin{Remark} \label{differential degree and smoothness 1} 
We can directly see from the definition that the smoothness at $P$ implies $d(C,P)=1$. Actually, the converse is also true, as we shall see in Proposition \ref{rL=1 imply smooth}. 
\end{Remark}

\begin{defn}
We denote the conductor of $\Gamma_{\Ou}$ by $c(\Ou)$, which is the largest positive integer $c\in \Gamma_{\Ou}$ such that $c-1$ is a gap of $\Gamma_{\Ou}$. We denote $c(C,P)$ for the conductor of $\Gamma_{C^{\overline{K}},P^{\overline{K}}}$.  \end{defn}

\begin{Remark}
It is proved in Lemma 2.1 pg. 313 of \cite{Bedoya} that $P_n$ is rational for some positive integer $n$.  Now we will use  Proposition \ref{conexion con el trabajo de Bedoya} to see the relation between our results in this section and some of their results in \cite{Bedoya}. Indeed we will show that 
$d(C,P)-1=\nu_n(dx_0^{p^n})$ where $x_0 \in \Ou_{C,P}$ is such that  $\Ou_{C_1,P_{1}}[x_0]=\Ou_{C,P}$. Moreover, it is important to emphasize that the number $\nu_{n}(dx_0^{p^n})$ is an important ingredient to compute the conductor $c(C,P)$ (cf. Theorem 2.3 in \cite{Bedoya}), or equivalently, to compute the
singularity degree (or $\delta$-invariant) $\delta(C,P):=\dim_{\overline{K}} \Ou_{\widetilde{C^{\overline{K}}},\widetilde{P^{\overline{K}}}}/\Ou_{C^{\overline{K}},P^{\overline{K}}}$ at $P$, since $2\delta(C,P)=c(C,P)$ (cf. Corollary 1.4 pg. 312 in \cite{Bedoya}).
\end{Remark}
\begin{prop}\label{conexion con el trabajo de Bedoya}
If $P_n$ is rational, then 
$d(C,P)-1=\nu_0(dx_0)=\nu_{n}(dx_0^{p^n})=\nu_{P_n}(dx_0^{p^n})$, where  $x_0 \in \Ou_{C,P}$ is such that  $\Ou_{C_1,P_{1}}[x_0]=\Ou_{C,P}$ and $\nu_{P_n}$ is the discrete valuation of the function field $K(C_n)|K$ with valuation ring $\Ou_{C_n,P_n}$.
\end{prop}
\begin{proof}
Suppose that $z$ is an uniformizing parameter of $\Ou_{\widetilde{C^{\overline{K}}},\widetilde{P^{\overline{K}}}}$. Since $\overline{K}$ is algebraically closed, it follows from the fundamental identity (cf.Theorem 3.1.11 in \cite{stichtenoth2009algebraic}) that $\nu_0$ is totally ramified over $\nu_n$ and hence $z^{p^n}$ is a uniformizing parameter of $\Ou_{\widetilde{C_n^{\overline{K}}},\widetilde{P_n^{\overline{K}}}}=(\Ou_{\widetilde{C^{\overline{K}}},\widetilde{P^{\overline{K}}}})^{p^n}$.
Let $D_z$ and $D_{z^{p^n}}$ be the derivations of 
$\Ou_{\widetilde{C^{\overline{K}}},\widetilde{P^{\overline{K}}}}$ and 
$\Ou_{\widetilde{C_n^{\overline{K}}},\widetilde{P_n^{\overline{K}}}}$, respectively, such that $D_z(z)=D_{z^{p^n}}(z^{p^n})=1$. From the definition, 
$\nu_0(dx_0)=\nu_0(D_z(x_0))=\dim_{\overline{K}} \coker (\alpha)=d(C,P)-1$.
Since $dx_0^{p^n}=D_{z^{p^n}}(x_0^{p^n})dz^{p^n}$ and $\nu_0$ is totally ramified over $\nu_n$ we just need to prove that $(D_z(x_0))^{p^n}=D_{z^{p^n}}(x_0^{p^n})$.

To do this we notice that 
$\Ou_{\widetilde{C^{\overline{K}}},\widetilde{P^{\overline{K}}}}=\Ou_{\widetilde{C_1^{\overline{K}}},\widetilde{P^{\overline{K}}}}[z]=\bigoplus^{p-1}_{i=0}\Ou_{\widetilde{C_1^{\overline{K}}},\widetilde{P^{\overline{K}}}}z^i$ implies that 
$\Ou_{\widetilde{C_n^{\overline{K}}},\widetilde{P_n^{\overline{K}}}}=(\Ou_{\widetilde{C^{\overline{K}}},\widetilde{P^{\overline{K}}}})^{p^n}=
\bigoplus^{p-1}_{i=0}(\Ou_{\widetilde{C^{\overline{K}}},\widetilde{P^{\overline{K}}}})^{p^{n+1}}(z^{p^n})^i = \bigoplus^{p-1}_{i=0}(\Ou_{\widetilde{C_n^{\overline{K}}},\widetilde{P_n^{\overline{K}}}})^{p}(z^{p^n})^i$ and, consequently, if $x_0=\sum_{i=0}^{p-1}a_i^pz^i$, then \[D_{z^{p^n}}(x_0^{p^n})= D_{z^{p^n}}\left(\sum_{i=0}^{p-1}(a_i^{p^n})^p(z^{p^n})^i\right) = 
\sum_{i=1}^{p^n-1}i(a_i^p)^{p^n}(z^{p^n})^{i-1} = D_z(x_0)^{p^n}.\]
To finish the proof we just need to observe that $\nu_n(f)=\nu_{P_n}(f)$ for every $f \in K(C_n)$. Indeed, for such $f$, we have $\nu_n(f)=e(\nu_n|\nu_{P_n})\nu_{P_n}(f)$, where $e(\cdot\,|\,\cdot)$ denotes the ramification index. But, from the proof of Lemma 1.5 in \cite{Bedoya}, $e(\nu_n|\nu_{P_n})=\deg P_n = 1$. 
\end{proof}


Now we will see how Theorem 2.3 in \cite{Bedoya} can be used to produce a recursive expression to compute the conductor in terms of differential degrees.  

\begin{thm}\label{Formula para el conductor}
Suppose that $n$ is a positive integer such that $P_n$ is rational.
Then \[c(C,P)= \sum_{0\leq i< n}(p-1)(d(C_i,P_i)-1)p^{i}\]
\end{thm}
\begin{proof}
From Theorem 2.3 in \cite{Bedoya} and Proposition \ref{conexion con el trabajo de Bedoya}, applied to $P_k$ instead of $P$, we have $c(C_k,P_k)=pc(C_{k+1},P_{k+1})+(p-1)(d(C_k,P_k)-1)$. By iterating this equality we obtain inductively $c(C,P)= \sum_{0\leq i< n}(p-1)(d(C_i,P_i)-1)p^{i}+p^{n}c(C_n,P_n).$ Notice that $P_n$ is smooth since it is rational. Therefore $c(C_n,P_n)=0$ from Corollary 1.4 in \cite{Bedoya}.
\end{proof}

\begin{lemm}\label{WrittingWritinn2}
If $x\in \Ou_{C,P}$ and $z$ is a uniformizing parameter of $\Ou_{\widetilde{C^{\overline{K}}},\widetilde{P^{\overline{K}}}}$, then for every $s>0$ there exists $Q(T)\in \overline{K}[T]$ of degree smaller than $s$, such that $\nu_0(x-Q(z))\geq s$.
\end{lemm}
\begin{proof}
If $n:=\nu_0(x)$, we have one of the following possibilities: $n\geq s$ or $n<s$. For the first case $Q(T)$ must be chosen equal to $0$. For the second case, we write $x=uz^n$ with $u=a_0+a_1z+\cdots+a_{s-n}z^{s-n}+\cdots$ being a unit of $\Ou_{\widetilde{C^{\overline{K}}},\widetilde{P^{\overline{K}}}}$ and hence $Q(T)$ must be $a_0T^n+a_1T^{n+1}+\cdots+a_{s-n-1}T^{s-1}$.
\end{proof}

\begin{lemm}\label{Lema de la existencia de x}
With the notation of Lemma \ref{WrittingWritinn2}, if $\nu_0(dx)=r-1$ then $x=Q(z^p)+z^ru$ with $Q(T)\in \overline{K}[T]$ of degree smaller than $r$ and $u\in \Ou_{C,P}^{\ast}$. In particular, $p$ does not divide $r$.
\end{lemm}
\begin{proof}
Lamma \ref{WrittingWritinn2} and hypothesis provide that there exists $Q(T)\in \overline{K}[T]$ of degree smaller than $r$ such that $x=Q(z)+z^{r}u$ with $u\in \Ou_{\widetilde{C^{\overline{K}}},\widetilde{P^{\overline{K}}}}^{\ast}$ and $dQ(z)= 0$. Hence   $Q(z)\in \overline{K}[z^p]$. On the other hand, if $p$ divides $r$, then $\nu_0(dx)=\nu_0(D_z(x))=\nu_0(z^rD_z(u))>r-1$.
\end{proof}

\begin{cor}\label{p no divide al grado diferenciasl}
    The differential degree $d(C,P)$ is not divided by $p$.
\end{cor}
\begin{proof}
    Follows from the above lemma and Proposition \ref{conexion con el trabajo de Bedoya}.
\end{proof}

Another consequence of the last lemma is the converse mentioned in Remark \ref{differential degree and smoothness 1}.

\begin{prop}\label{rL=1 imply smooth}
If $d(C,P)=1$, then $P$ is smooth.
\end{prop}
\begin{proof}
If $d(C,P)=1$ then  Lemma \ref{Lema de la existencia de x} together with Proposition \ref{conexion con el trabajo de Bedoya} say that there exists $x\in \Ou_{C,P}$ such that $x=a_0+z$, where $z$ is an uniformizing parameter of $\Ou_{\widetilde{C^{\overline{K}}},\widetilde{P^{\overline{K}}}}$ and $a_0\in \overline{K}$. Therefore $z\in \overline{K}\Ou_{C,P}$ and, consequently,  $\Ou_{C^{\overline{K}},P^{\overline{K}}}=\overline{K}\Ou_{C,P}=\Ou_{\widetilde{C^{\overline{K}}},\widetilde{P^{\overline{K}}}}$. \end{proof}

\begin{lemm}\label{lemma prolegomeno a forma linear de karl}
Let us consider $x_0\in \overline{K}\Ou_{C,P}$ such that $\overline{K}\Ou_{C,P}=\overline{K}\Ou_{C_1,P_1}[x_0]$. Then the following properties hold.
\begin{enumerate}
    \item There is a derivation $D_{x_0}:\overline{K}\Ou_{C,P}\to \overline{K}\Ou_{C,P}$ such that   $D_{x_0}^{p-1}(\overline{K}\Ou_{C,P})\subset\overline{K}\Ou_{C_1,P_1}$ and $D_{x_0}(x_0)=1$;
    \item For every $w\in \overline{K}\Ou_{C,P}$ we have $\nu_0(D_{x_0}(w))\geq \nu_0(w)-d(C,P)$, with equality holding if and only if $p$ does not divide $\nu_0(w)$;
   \end{enumerate}

\end{lemm}
\begin{proof}
(1) If $z$ is an uniformizing parameter of $\Ou_{\widetilde{C^{\overline{K}}},\widetilde{P^{\overline{K}}}}$, we define $D_{x_0}=D
_{z}/D_{z}(x_0)$ where $D_{z}$ is the unique derivation of  $\Ou_{\widetilde{C^{\overline{K}}},\widetilde{P^{\overline{K}}}}$ such that $D_z(z)=1$. 
Notice that $D_{x_0}(\overline{K}\Ou_{C_1,P_1})=0$ and $D_{x_0}(x_0)=1$. Given $w\in \Ou_{C,P}$ we may write $w=\sum_{i=0}^{p-1} a_ix_0^i$ where $a_i\in \overline{K}\Ou_{C_1,P_1}$. Therefore  $D_{x_0}(w)=\sum_{i=1}^{p-1} ia_ix_0^{i-1}$, which in turn implies that $D_{x_0}(\overline{K}\Ou_{C,P})\subset \overline{K}\Ou_{C,P}$ and  $D_{x_0}^{p-1}(\overline{K}\Ou_{C,P})\subset \overline{K}\Ou_{C_1,P_1}$.  

(2) From the definition of $D_{x_0}$ and since $\nu_0(dx_0)=d(C,P)-1$ we have $\nu_0(D_{x_0}(w))=\nu_0(D_{z}(w))-d(C,P)+1.$ Therefore $ \nu_0(D_{x_0}(w)) = \nu_0(dw) -d(C,P)+1 \geq \nu_0(w) -d(C,P)$, and the identity holds if and only if $\nu_0(w)$ is coprime with $p$. 
\end{proof}

\begin{thm}\label{descomposicion simple del semigrupo}
If $d(C_1,P_1)=1$ then \[\Gamma_{C^{\overline{K}},P^{\overline{K}}}=d(C,P)\mathbb{Z}_{\geq 0}+p\mathbb{Z}_{\geq 0}.\]
\end{thm}
\begin{proof}
As we have already seen, if $x_0\in \Ou_{C,P}$ is such that $\Ou_{C,P}=\Ou_{C_1,P_1}[x_0]$, then $\nu_0(dx_0)=d(C,P)-1$ and, from Lemma \ref{Lema de la existencia de x},    $x_0=Q+y$ with $\nu_0(y)=d(C,P)$ and $Q\in \Ou_{\widetilde{C_1^{\overline{K}}},\widetilde{P_1^{\overline{K}}}}$. Since $d(C_1,P_1)=1$ we obtain $\left(\Ou_{\widetilde{C^{\overline{K}}},\widetilde{P^{\overline{K}}}}\right)^p=\Ou_{\widetilde{C_1^{\overline{K}}},\widetilde{P_1^{\overline{K}}}}=\overline{K}\Ou_{C_1,P_1}$ and, consequently, if $z$ is the uniformizing parameter of $\Ou_{\widetilde{C^{\overline{K}}},\widetilde{P^{\overline{K}}}}$, then  $z^p, y\in \overline{K}\Ou_{C,P} \mbox{ and } d(C,P)\mathbb{Z}_{\geq 0}+p\mathbb{Z}_{\geq 0}\subset \Gamma_{C^{\overline{K}},P^{\overline{K}}}.$ 

Let $D_{x_0}:\overline{K}\Ou_{C,P}\to \overline{K}\Ou_{C,P}$ be the derivation defined in Lemma \ref{lemma prolegomeno a forma linear de karl}.
Given $x\in \overline{K}\Ou_{C,P}$ if $p$ does not divide $\nu_{0}(x)$, then $\nu_{0}(x)=n_xd(C,P)\mod(p)$ for some $n_x\in \{1,\ldots,p-1\}$. Therefore, after applying item (2) of the previous lemma $n_x$ times, we get that $\nu_{0}(D_{x_0}^{n_x}(x))=\nu_{0}(x)-n_xd(C,P)$ is a multiple of $p$ and we conclude that $\nu_0(x)\in d(C,P)\mathbb{Z}_{\geq 0}+p\mathbb{Z}_{\geq 0}$.
\end{proof}

\begin{prop}\label{grado de un punto divide deferencia de grados diferenciales}
We have that 
$[\textbf{k}(C_i,P_i):K]\mid d(C,P)-d(C_i,P_i)$, for every positive integer $i$.
\end{prop}
\begin{proof}
Take $x\in \Ou_{C,P}$ and $x_i\in \Ou_{C_i,P_i}$ such that $\nu(dx)=d(C,P)-1$ and  $\nu_i(dx_i)=d(C_i,P_i)-1$. With the techniques developed in the proof of Proposition \ref{conexion con el trabajo de Bedoya} we may see that $dx^{p^i}=fdx_i$ with $f\in \Ou_{C_i,P_i}$. From the proof of Lemma 1.5 in \cite{Bedoya} we can deduce that $\nu_i(\Ou_{C_i,P_i})=[\textbf{k}(C_i,P_i):K]\mathbb{Z}$, which in turn implies  that $[\textbf{k}(C_i,P_i):K]\mid \nu_i(f)=\nu_i(dx^{p^i})-\nu_i(dx_i)=d(C,P)-d(C_i,P_i)$. 
\end{proof}
\begin{cor}\label{caso muy general de uso}
If $d(C,P)-1$ is not divided by $p$ and $m$ is an integer such that $P_m$ is smooth, then  $P_m$ is also rational.
\end{cor}
\begin{proof}
We have that $d(C_m,P_m)=1$ and then $[\textbf{k}(C_m,P_m):K]\mid d(C,P)-1$, from Proposition \ref{grado de un punto divide deferencia de grados diferenciales}. Since $p$ does not divide  $d(C,P)-1$ then the $p$-power $[\textbf{k}(C_m,P_m):K]$ is coprime with $p$, that is, $[\textbf{k}(C_m,P_m):K]=1$ and $P_m$ is rational.
\end{proof}
\begin{thm} \label{cota para o grau do ponto}
We have that
 $[\textbf{k}(C,P):K]\mid c(C,P)+d(C,P)-1$
\end{thm}
\begin{proof}
    Since $[\textbf{k}(C_i,P_i):K]$ divides $d(C,P)-d(C_i,P_i)$ and $[\textbf{k}(C,P):\textbf{k}(C_i,P_i)]$ divides $p^i$ then 
$[\textbf{k}(C,P):K]=[\textbf{k}(C,P):\textbf{k}(C_i,P_i)][\textbf{k}(C_i,P_i):K]\mid (d(C,P)-d(C_i,P_i))p^i.$ Therefore:
$$[\textbf{k}(C,P):K]\mid \sum_{i=0}^n (p-1)(d(C,P)-d(C_i,P_i))p^i=d(C,P)(p^{n+1}-1)-\sum_{i=0}^n (p-1)d(C_i,P_i)p^i$$
where $n$ is such that $P_n$ is rational.
    Notice that Theorem \ref{Formula para el conductor}, applied to $P_{n+1}$, says that $$p^{n+1}-1-\sum_{i=0}^n (p-1)d(C_i,P_i)p^i=\sum_{i=0}^n (p-1)(1-d(C_i,P_i))p^i=-c(C,P)$$
    Then $[\textbf{k}(C,P):K]\mid p^{n+1}(d(C,P)-1) -(c(C,P)+d(C,P)-1)$ and the proof follows from the fact that $[\textbf{k}(C,P):K]\mid p^n$.
\end{proof}
\begin{cor}\label{estudio del grado de singularidad}
    If the singularity degree at the regular point $P$ of $C$ is smaller than $(p-1)^2/2$ then:
\begin{enumerate}
    \item $1<d(C,P)<p$ and $d(C_1,P_1)=1$;
    \item $P_1$ is a rational point of $C_1$;
    \item $[\textbf{k}(C,P):K]=p$.
    \item $\dim_{\overline{K}}\m_{C^{\overline{K}},P^{\overline{K}}}/\m_{C^{\overline{K}},P^{\overline{K}}}^2=2$
    \item $\Gamma_{C^{\overline{K}},P^{\overline{K}}}$ is generated by $d(C,P)$ and $p$.
\end{enumerate}
\end{cor}
\begin{proof}
   $(1)$ and $(2)$ Theorem \ref{Formula para el conductor} says that $0<\sum_i (p-1)(d(C_i,P_i)-1)p^i<(p-1)^2$ and, consequently, $d(C_1,P_1)=1$ and $1<d(C,P)<p$. In particular $p$ is coprime with $d(C,P)-1$ and Corollary \ref{caso muy general de uso} together with Proposition \ref{rL=1 imply smooth} imply that $P_1$ is rational. $(3)$ The above theorem provides that $[\textbf{k}(C,P):K]\leq (p-1)p$ and, hence, $[\textbf{k}(C,P):K]=p$, since $[\textbf{k}(C,P):K]>1$ is a power of $p$. $(4)$ Follows from
Proposition \ref{basis canonica de embebimento}, because $[\textbf{k}(C,P):K]=p$ and 
$P^{\overline{K}}$ singular imply, respectively, that $\dim_{\textbf{k}(C,P)} \Omega_{\textbf{k}(C,P)/K}=1$  and  $\dim_{\overline{K}}\m_{C^{\overline{K}},P^{\overline{K}}}/\m_{C^{\overline{K}},P^{\overline{K}}}^2>1$. $(5)$ Follows from $(1)$ and Theorem \ref{descomposicion simple del semigrupo}.
\end{proof}

\section{Invertible sheaves on non-smooth regular curves}\label{Sectionn 4}

In this section we will describe some properties of suitable sheaves on non-smooth regular curves, which will be important to explore the geometry of the singularities appearing after their base changes. To simplify the notation, we will use $C^{k}$ and $P^k$ instead of the previously introduced notation $C^{K^{p^{-k}}}$ and $P^{K^{p^{-k}}}$, respectively, where $k$ is any positive integer.

Notice that the absolute Frobenious morphism provides the isomorphism $\widetilde{C^k}\simeq C_k$ of curves defined over $K^{p^{-k}}$ and $K$, respectively, for every $k\geq 0$.  Let $\mathfrak{b}_k:\widetilde{C^k}\to C$ (resp. $\mathfrak{b}_{\infty}:\widetilde{C^{\overline{K}}}\to C$) be the composition of the normalization morphism $\mathfrak{n}_k:\widetilde{C^k}\to C^k$ (resp. $\mathfrak{n}_{\infty}:\widetilde{C^{\overline{K}}}\to C^{\overline{K}}$) with the base change morphism $\mathfrak{e}_k:C^k\to C$ (resp. $\mathfrak{e}_{\infty}:C^{\overline{K}}\to C$). 

Suppose that $n$ is chosen in such a way that the points of  $\widetilde{C^n}$, lying over the non-smooth points of $C$, are rational. The existence of such $n$ is guaranteed by Lemma 2.1 pg. 313 of \cite{Bedoya}. We define the divisors $D'=\sum_{P\in C}c(C,P)\widetilde{P^n}$ and $D''=\sum_{P\in C}(d(C,P)-1)\widetilde{P^n}$ of $\widetilde{C^n}$ and, according to Theorem \ref{cota para o grau do ponto}, we define the divisor  $D=\sum_{P\in C} a_PP$ of $C$,  where $a_P:=\frac{c(C,P)+d(C,P)-1}{[\textbf{k}(C,P):K]}$. Notice that the support of $D$ consists only at non-smooth points of $C$ and $D'+D''=\mathfrak{b}_n^{\ast}D$.

\begin{thm}\label{Lema a cerca de igualdad de divisores}
Let $\mathcal{L}_1$ and $\mathcal{L}_2$ be two invertible sheaves of $C$. If $n$ is as above and  $\mathfrak{b}_n^{\ast}\mathcal{L}_1\simeq\mathfrak{b}_n^{\ast}\mathcal{L}_2$,  then $\mathcal{L}_1\simeq \mathcal{L}_2$.
\end{thm}
\begin{proof}
After tensorizing by an ample and invertible sheaf we can assume that $H^{0}(C,\mathcal{L}_i)$ has a dimension greater than $\sum_{P\in C} c(C,P)/2$. Notice  we have the following exact sequence.
\[0\to H^{0}(C^n,\mathfrak{e}_n^{\ast}(\mathcal{L}_i))\to H^{0}(\widetilde{C^n},\mathfrak{b}_n^{\ast}(\mathcal{L}_i))\to H^0\left(C^n,\frac{\Ou_{\widetilde{C^n}}}{\Ou_{C^n}}\right)\]
Therefore the codimiension of $H^{0}(C^n,\mathfrak{e}_n^{\ast}(\mathcal{L}_i))$ in $H^{0}(\widetilde{C^n},\mathfrak{b}_n^{\ast}(\mathcal{L}_i))$ is at most $\sum_{P\in C} c(C,P)/2$ and, from the assumption in the beginning, there is $s\in \bigcap_{i=1,2} H^{0}(C^n,\mathfrak{e}_n^{\ast}(\mathcal{L}_i))$. If $R$ is the Weil divisor induced by $s$, then $\mathfrak{e}_n^{\ast}(\mathcal{L}_1)=\Ou_{C^n}(R)=\mathfrak{e}_n^{\ast}(\mathcal{L}_1)$ and, therefore, $\mathfrak{e}_n^{\ast}(\mathcal{L}_1\otimes \mathcal{L}_2^{\otimes (-1)})\simeq \Ou_{C^n}$. Proposition \ref{Prefacio a eigualdad clave} applied to  $\mathfrak{e}_n^{\ast}(\mathcal{L}_1\otimes \mathcal{L}_2^{\otimes (-1)})$ says that $\mathfrak{e}^{\ast}_{n-1}(\mathcal{L}_1\otimes \mathcal{L}_2^{\otimes (-1)})$ has a non-vanishing global section. After applying recursively the above argument we obtain a non-vanishing global section for $\mathcal{L}_1\otimes \mathcal{L}_2^{\otimes (-1)}$ and, analogously, for $\mathcal{L}_2\otimes \mathcal{L}_1^{\otimes (-1)}$. So $\mathcal{L}_1\otimes \mathcal{L}_2^{\otimes (-1)}=\Ou_C(S)$ and $\mathcal{L}_2\otimes \mathcal{L}_1^{\otimes (-1)}=\Ou_C(T)$, where $S$ and $T$ are positive divisors of $C$ associated to the respective global sections. Since, $\Ou_C(S)\otimes \Ou_C(T) \simeq \Ou_C$, with $S,T \geq 0$, we have $S=T=0$ and $\mathcal{L}_1\simeq \mathcal{L}_2$.
\end{proof}

We denote the canonical sheaf of a curve $E$ over $K$ by $\omega_{E/K}$. If $L|K$ is an algebraic field extension, then, from \cite{Liu} Theorem 4.9 pg. 252, we have that  $\omega_{E^L/L}$ is the pullback of $\omega_{E/K}$ by the base extension map $E^L \to E$.

\begin{prop}\label{lema para la inclusion de haces canonicos}
    If $n$ is as above, then there is an inclusion of sheaves:  \[(\mathfrak{n}_n)_{\ast}(\omega_{\widetilde{C^n}/K^{p^{-n}}})\hookrightarrow \omega_{C^n/K^{p^{-n}}}.\]
\end{prop}
\begin{proof}
    From \cite{Liu}, Remark 3.28 in pg. 282, $H^1(\widetilde{C^{n}},\omega_{\widetilde{C^n}/K^{p^{-n}}})\neq 0$. On the other hand, since  $\mathfrak{n}_n$ is an affine morphism and $\widetilde{C^{n}}$ is separated, then we may use Exercise 2.3, pg. 191, of \cite{Liu} to conclude that $H^1(C^{n},(\mathfrak{n}_n)_{\ast}(\omega_{\widetilde{C^{n}}/K^{p^{-n}}}))\neq 0$. Moreover Grothendieck's duality says that $$H^0(C^{n}, \omega_{C^n/K^{p^{-n}}}\otimes(\mathfrak{n}_n)_{\ast}(\omega_{\widetilde{C^{n}}/K^{p^{-n}}})^{\vee}) = H^1(C^{n},(\mathfrak{n}_n)_{\ast}(\omega_{\widetilde{C^{n}}/K^{p^{-n}}})) \neq 0,$$  implying the existence of a morphism
$(\mathfrak{n}_n)_{\ast}(\omega_{\widetilde{C^n}/K^{p^{-n}}})\to \omega_{C^n/K^{p^{-n}}}.$

It is an inclusion because otherwise, the image of $(\mathfrak{n}_n)_{\ast}(\omega_{\widetilde{C^{n}}/K^{p^{-n}}})$ would produce a torsion subsheaf of $\omega_{C^n/K^{p^{-n}}}$, which is impossible because $C^n$ is a base change of a regular curve $C$, hence it is Gorenstein (see \cite{Sthor2}, Theorem 1.1 or \cite{watanabe1969tensor}). 
\end{proof}

\begin{thm}\label{segundo lema de divisores canonicos}

If $n$ is as above, then the sheaf $\Omega_{C/C_1}$ is invertible and we have the following isomorphism of sheaves.
   
   \begin{itemize}
    \item $\Omega_{C/C_1}\otimes\Ou_C(D)\simeq\omega_{C/K}$.
    \item $\mathfrak{b}_n^{\ast}(\Omega_{C/C_1})\otimes \Ou_{\widetilde{C^{n}}}(D'')\simeq \omega_{\widetilde{C^{n}}/K^{p^{-n}}}$.
    \item $\omega_{\widetilde{C^{n}}/K^{p^{-n}}}\otimes\Ou_{\widetilde{C^{n}}}(D')\simeq \mathfrak{b}_n^{\ast}(\omega_{C/K})$.
\end{itemize}
   
\end{thm}
\begin{proof}
Firstly we prove that $\Omega_{C/C_1}$ is an invertible sheaf. Indeed we only need to prove that $\Omega_{C/C_1,P}$ is generated by only one element over each non-smooth point $P$ of $C$. From Proposition \ref{conexion con el trabajo de Bedoya}, $\Ou_{C,P}=\Ou_{C_1,P_1}[x_0]$ for some $x_0$ such that $\nu_{0}(dx_0)=d(C,P)-1$ and then  $\Omega_{C/C_1,P}=\Ou_{C,P}dx_0$, which is invertible. 

Let's prove now the second isomorphism. Since $\widetilde{C^{n}}$ is smooth, we have that $\Omega_{\widetilde{C^n}/K^{p^{-n}}} \simeq \omega_{\widetilde{C^n}/K^{p^{-n}}}$. Moreover, if $\nu_{\widetilde{P^n}}$ is the valuation of $K^{p^{-n}}K(C)|K^{p^{-n}}$ with valuation ring $\Ou_{\widetilde{C^n},\widetilde{P^n}}$, then $\nu_{\widetilde{P^n}}(dx_0) = \nu_{0}(dx_0)=d(C,P)-1$, because $\widetilde{P^n}$ is rational. Therefore, the inclusion of sheaves
$\mathfrak{b}_n^{\ast}(\Omega_{C/C_1})\otimes \Ou_{\widetilde{C^n}}(D'') \hookrightarrow\omega_{\widetilde{C^n}/K^{p^{-n}}}$ 
is locally an isomorphism, because the image consists of $1$-forms having non-negative value under $\nu_{\widetilde{P^n}}$. Hence it is an isomorphism.

Now, to prove the third isomorphism, 
we notice that Proposition \ref{lema para la inclusion de haces canonicos} says that 
$\omega_{\widetilde{C^n}/K^{p^{-n}},\widetilde{P^n}}\subset \omega_{C^{n}/K^{p^{-n}},P^n}$.
Since $\omega_{C^n/K^{p^{-n}},P^n}$ and $\omega_{\widetilde{C^n}/K^{p^{-n}},\widetilde{P^n}}$ are invertible sheaves, there are $f,g\in \omega_{C^n/K^{p^{-n}},P^n}$ such that 
$\omega_{\widetilde{C^n}/K^{p^{-n}},\widetilde{P^n}}=\Ou_{\widetilde{C^{n}},\widetilde{P^n}}f \mbox{ and } \omega_{C^{n}/K^{p^{-n}},P^n}=\Ou_{C^{n},P^{n}}g.$ In particular $fg^{-1}\Ou_{\widetilde{C^n},\widetilde{P^{n}}}\subset \Ou_{C^{n},P^{n}}$ and, from the definition of the conductor and the rationality of $P^{n}$, we have $\nu_{\widetilde{P^n}}(fg^{-1})\geq c(C,P)$. Using the same ``local x global" argument, to conclude the second isomorphism, we get
$\omega_{\widetilde{C^n}/K^{p^{-n}}}\simeq\mathfrak{b}_n^{\ast}(\omega_{C/K})\otimes\Ou_{\widetilde{C^n}}(-R)$
where $R\geq D'$ is a divisor, locally (at $\widetilde{P^n}$) given by the ideal generated by $fg^{-1}$. Since $\widetilde{C^n}$ is smooth, we obtain  \[\begin{array}{lllll}
    \deg(\mathfrak{b}_n^{\ast}(\omega_{C/K})\otimes\Ou_{\widetilde{C^n}}(-D'))&=&2p_a(C)-2-\deg(D') &=&2p_g(C)-2 \\ &=&\deg(\omega_{\widetilde{C^n}/K^{p^{-n}}})&=&2p_a(C)-2-\deg(R)
\end{array}\] and, consequently, $R=D$.

To prove the first isomorphism we use the identity $\mathfrak{b}_n(D)=D'+D''$ together with the second and third isomorphisms to obtain $\mathfrak{b}_n^{\ast}(\Omega_{C/C_1}\otimes\Ou_C(D))\simeq\mathfrak{b}_n^{\ast}(\omega_{C/K})$. Now we just need to use Theorem \ref{Lema a cerca de igualdad de divisores} to conclude the proof.
\end{proof}

    


\section{Non $\Py^2_K$-invariant rational points in $\Py^2_{K^{1/3}}$}\label{Sectionn 5}
In this section we will suppose that $K$ is a field of characteristic three.
We will denote the projective planes  $\Py^2_K$ and $\Py^2_{K^{1/3}}$ by $X$ and $Y$, respectively, and we will also denote the natural base change by $\pi:Y\to X$. Let $\Inv(Y/X)$ be the category considered in Definition \ref{definicion MIC}. Remember that there is a  natural connection $\nabla(n):\Ou_{Y}(n)\to \Omega_{Y/X}\otimes_{\Ou_{Y}}\Ou_{Y}(n)$,  which was obtained in Proposition \ref{definicion del funtor}. In this section we will use such connections to study some geometric properties of $K^{1/3}$-rational points of $Y$, namely the ``$1$-type" introduced in Definition \ref{definicion de type} bellow, to build normal forms of non-smooth regular curves in next chapter. For this purpose, suppose that $P$ is a point of $X$ and $P^1$ is the point of $Y$ lying over $P$. 

\begin{defn}\label{definicion de type}

We say that a closed subscheme $Z$ of $Y$ is of $m$-type $k$ when $k$ is the smallest degree of an $X$-invariant $m$-dimensional subvariety of $Y$, whose ideal sheaf is contained in the ideal sheaf of $Z$.
\end{defn} 
\begin{defn}
    We say that two sets of $X$-invariant subschemes of $Y$, namely $\{V_1,\ldots, V_n\}$ and $\{W_1,\ldots, W_n\}$, are $X$-equivalent if there is an automorphism $T$ in $X$ such that $T\otimes_X 1_Y(V_i)=W_i$.

\end{defn}
\begin{prop}\label{auto lines send}
   Each subset of three different $X$-invariant lines $\{M_1,M_2,M_3\}$ of $Y$ is $X$-equivalent to $\{x,y,z\}$, where $x,y,z$ is a basis of $H^0(Y,\Ou_{Y}(1))^{\nabla(1)}=H^0(X,\Ou_{X}(1))$.
\end{prop}
\begin{proof}
There exist a line $N_i$ in $X$ such that $N_i^{K^{1/3}}=M_i$. So take the automorphism $T$ of $X$ sending $N_1$ to $x$, $N_2$ to $y$ and $N_3$ to $z$. Then $T\otimes_X 1_Y$ is the required automorphism.
\end{proof}

\begin{prop}\label{tipo 1 tipo 2 prueba}
If $P$ is a point of $X$ such that its residue field  $\textbf{k}(X,P)$ is a degree three extension of $K$, then $P^1$ is of $1$-type $1$ or of $1$-type $2$.  
\end{prop}
\begin{proof}
Let us consider a basis $x,y,z$ of $H^0(Y,\Ou_{Y}(1))^{\nabla(1)}=H^0(X,\Ou_{X}(1))$. We may assume that, over such basis, $P^1$ is represented by the point  $(a:b:1)$ with $a, b \in K^{1/3}$. Indeed, from the hypothesis, $\textbf{k}(X,P)=K[\alpha ]$ with $\alpha \in K^{1/3}\setminus K$. On the other hand, $\textbf{k}(Y,P^1)$ is the quotient of $\textbf{k}(X,P)\otimes_K K^{1/3}$ by its maximal ideal, which is nothing else than $K^{1/3}[\alpha]=K^{1/3}$.

Using that $P^1$ is non $K$-invariant we can assume $a\in K^{1/3}\setminus K$. Moreover, since  $\overline{\frac{x}{z}}(P)=a$ and $\overline{\frac{y}{z}}(P)=b$, as elements in $\textbf{k}(X,P) \subset \textbf{k}(Y,P^1)$, we obtain that $\textbf{k}(X,P)=K[a]$. Hence $b=f(a)$ for some polynomial $f(T)\in K[T]$ of degree one or two. If $f(T)=c_0+c_1T \in K[T]$ (respectively, $f(T)=c_0+c_1T+c_2T^2 \in K[T]$), then the curve $y-c_0z-c_1x$ (respectively, $yz-c_0z^2-c_1xz-c_2x^2$) contains $P^1$. 
\end{proof}

\begin{Remark}
    Notice that the point $P^1\in Y$, lying over $P \in X$, may not be $X$-invariant. Indeed, the base extension of the ideal defining $P$ may not be reduced and, consequently, may not be the ideal defining $P^1$. In this way, the next proposition gives a numerical condition to obtain the $X$-invariance at $P^1$.
\end{Remark}

\begin{prop}\label{interseccion de curvas invariantes}
Suppose that  $C_1$ and $C_2$ are two $X$-invariant curves in $Y$. If the intersection number between $C_1^{\overline{K}}$ and $C_2^{\overline{K}}$ at $P^{\overline{K}}$ is smaller than three, then $P$ is a $K$-rational point of $X$. In particular, $P^1$ is an $X$-invariant point of $Y$.
\end{prop}

\begin{proof}
Suppose that $x_i$ is a local equation in $\Ou_{X,P}$ inducing $C_i$ for $i=1,2$. The hypothesis says that  \[ \dim_K \Ou_{X,P}/\langle x_1,x_2 \rangle=\dim_{\overline{K}} \Ou_{X^{\overline{K}},P^{\overline{K}}}/\langle x_1,x_2 \rangle\leq 2.\] In particular $\dim_K \Ou_{X,P}/\m_{X,P}\leq 2$. Thus $ \Ou_{X,P}/\m_{X,P}$ is a separable extension of $K$ and, since $K$ is separably closed, we conclude that $\Ou_{X,P}/\m_{X,P}=K$. Hence, if the ideal sheaf at $P$ is locally given by $<\tilde{x}_1-a_1,\tilde{x}_2-a_2>$ where $\tilde{x}_1,\tilde{x}_2$ are local equations of $X$ at $P$ and $a_1,a_2 \in K$, then the ideal sheaf at $P^1$ is locally given by the base extension of $<\tilde{x}_1-a_1,\tilde{x}_2-a_2>$. 
\end{proof}
\begin{cor}\label{no existen lineas pasando por puntos tipo 2}
If a $K^{1/3}$-rational and non $X$-invariant point $P^1$ belongs to an irreducible $X$-invariant conic $C$, then there is no $X$-invariant line passing through $P^1$. 
\end{cor}
\begin{proof}
Suppose that $P^1\in C\cap L$ where $C$ is an $X$-invariant and irreducible conic and $L$ is an $X$-invariant line. Proposition \ref{interseccion de curvas invariantes} says that the intersection number between $C^{\overline{K}}$ and $L^{\overline{K}}$ at $P^{\overline{K}}$ is greater than three, contradicting  Bezout's theorem.   
\end{proof}

\begin{defn} \label{definision del X cociente}
Let $L$ be a line in $Y$ defined over $K^{1/3}$ and let $L^3$ be the $X$-invariant curve generated by the cubic of $L$. We take $L_1$ as the quotient of $L^3$ by the foliation of base change restricted to $L^3$, or equivalently, the cubic of $X$ whose base change is $L^3$. We call $L_1$ by the $X$-quotient of $L^3$.
\end{defn}
\begin{prop}\label{regular points of the Xquotient}
Suppose that $L$ is a non $X$-invariant line of $Y$ defined over $K^{1/3}$. If $L_1$ is the $X$-quotient of $L^3$, then $L_1$ is an irreducible curve of $X$ having at most one $K$-rational point. Moreover, if $P$ is a non $K$-rational point of $L_1$ then $\Ou_{L_1,P}$  is a regular ring. 
\end{prop}
\begin{proof}
    Let $U$ be an affine open subset of $Y$, isomorphic to $\spec(K^{1/3}[x,y])$. Suppose that $L^3$ is defined by $a^3x^3+b^3y^3+c^3$ with $a,b,c\in K^{1/3}$. Since $L$ is non-invariant, $a$, $b$ or $c$ do not belong to $K$ and, therefore, the uniqueness of factorization in  $K^{1/3}[x,y]$ implies the irreducibility of $a^3x^3+b^3y^3+c^3$ in $K[x,y]$. On the other hand, it holds for every affine open subset of $Y$ as above, implying that $L_1$ is irreducible and hence $a\neq 0$ or $b\neq 0$. Now, let us suppose that there are two $K$-rational points of $L_1$. In this case, $L$ is the base extension of the line through these $K$-rational points, because it is uniquely determined by the $K^{1/3}$-rational points lying over the first ones. So $L$ should be invariant, contradicting the hypothesis.
        
    Now we will prove that $L_1$ is regular at $P$, provided that $P$ is a non $K$-rational point. Suppose that $U$ is an affine open subset of $Y$, isomorphic to $\spec(K^{1/3}[x,y])$, such that $P\in U$ and $L$ is defined by $a^3x^3+b^3y^3+c^3$. Moreover, suppose that $P^{\overline{K}}$ is locally defined by $(a_1,b_1)$ with $a_1,b_1\in \overline{K}$ satisfying, without loss of generality, that $b_1\not \in K$.  So we obtain $L^3=a^3(x-a_1)^3+b^3(y-b_1)^3$. If $a=0$ then $b\neq 0$ and $\m_{L_1,P}$ is generated by $(y-b_1)^{3^j}$, where $j$ is the minimum integer such that $(y-b_1)^{3^j}\in K[x,y]$.  If $a$ is different from zero, then $\m_{L_1,P}$ is generated by $(y-b_1)^{3^j}$, if $b\neq 0$, or $(x-a_1)^{3^i}$, if $b=0$ (implying that $a_1\not\in K$), where $i$ and $j$ are the minimum integers such that $(x-a_1)^{3^i},(y-b_1)^{3^j}\in K[x,y]$. In all of these cases $\m_{L_1,P}$ is generated by one element and, hence, $P$ is regular. 
    \end{proof}
\begin{prop}\label{como se ven los divisores de L_1}
    If $L$ is a non $X$-invariant line of $Y$ defined over $K^{1/3}$, $L_1$ is the $X$-quotient of $L^3$ and $P$ is a $K$-rational point of $L_1$, then the following three sentences hold.
    \begin{enumerate}
        \item The ring $\Ou_{L_1,Q}$ is regular for every $Q \neq P$.
        \item The ideal sheaf $\mathcal{I}_{L_1,Q}$ of $Q$, as a point of $L_1$, is invertible for every $Q \neq P$.
        \item If $D$ and $D'$ are two Cartier divisors of $L_1$ such that $[D]=[D']$ (where $[D]$ is the Weil divisor associated to $D$) and such that $P$ does not belong to the support of $D-D'$,  then $D =D'$. 
    \end{enumerate} 
    \end{prop}
\begin{proof}
    We start proving the first two sentences. If $Q\in L_1$ is different from $P$ and $Q$ is $K$-rational, then $L$ is $X$-invariant, which is an absurd. Therefore $Q$ is not $K$-rational and $\Ou_{L_1,Q}$ is regular by Proposition \ref{regular points of the Xquotient}. In particular, $\mathcal{I}_{L_1,Q}$ is generated by a unique element in a neighborhood of $Q$ and, consequently, it is invertible. 
    
    To prove the last sentence, if $[D-D']=0$ then Proposition \ref{regular points of the Xquotient} says that $(D-D')\mid_{U}=1$ for every neighborhood not containing $P$. Since $P$ does not belong to the support of $D-D'$, there exists an open subset $V$ containing $P$, such that  $(D-D')\mid_{V}=1$. Therefore, $D=D'$.
\end{proof}

\begin{defn}
 Suppose that $L$ is a non $X$-invariant line defined over $K^{1/3}$ and $L_1$ is the $X$-quotient of $L^3$. If we take a set of non $K$-rational points $P_1,\ldots,P_s$ of $L_1$, then we denote by $\Ou_{L_1}(n_1P_1+\ldots+n_sP_s)$ the invertible sheaf $\mathcal{I}_{L_1,P_1}^{\otimes -n_1}\otimes \ldots \otimes \mathcal{I}_{L_1,P_s}^{\otimes -n_s}$ 
 where $n_1,\ldots,n_s$ are integers and $\mathcal{I}_{L_1,P_j}$ is the ideal sheaf at $P_j$ as a point of $L_1$. Such sheaf is invertible by Proposition \ref{como se ven los divisores de L_1}.
 \end{defn}

\begin{prop}\label{segundo isomorfismo loco}
 Let $L$ be a non $X$-invariant line of $Y$ defined over $K^{1/3}$ and $L_1$ be the $X$-quotient of $L^3$. Let us assume that there are two different non $K$-rational points $P$ and $Q$ of $L_1$ whose base extensions $P^1$ and $Q^1$ are $K^{1/3}$-rational points and $P^1$ or $Q^1$ is of $1$-type $2$. If $\Ou_{L_1}(P+Q)\simeq \Ou_{L_1}(2)$ then $\{P^1,Q^1\}$ is $X$-equivalent to $\{(t: t^2:1),(u: u^2:1)\}$ with $t,u\in K^{1/3}$. 
\end{prop}
\begin{proof}
    Notice that $\Ou_X(2)$ is very ample and then the morphism $H^0(X,\Ou_X(2))\to H^0(L_1,\Ou_{L_1}(2))$ is an epimorphism. Therefore the hypothesis $\Ou_{L_1}(P+Q)\simeq \Ou_{L_1}(2)$ says that there exist a section $s\in H^0(X,\Ou_X(2))$ with $P$ and $Q$ belonging to the support of $\Div(s)$. If $F$ is the conic of $X$ defined by $s$ then $P$ and $Q$ belong to $F$. We have that $F$ is irreducible because $P^1$ or $Q^1$ is of $1$-type $2$.
    
    Suppose that $L_1$ and $L_2$ are two $X$-invariant tangents lines to $F^{K^{1/3}}$ through two $X$-invariant and different points $R_1$ and $R_2$ of $F^{K^{1/3}}$, respectively. If $L_3$ is the $X$-invariant line  containing $R_1$ and $R_2$, then $L_3^2$ and $F^{K^{1/3}}$ belong to the pencil of conics that are tangent to $L_1$ and $L_2$ at $R_1$ and $R_2$ respectively. So we can find a $\lambda\in K$ such that $F=L_1L_2-\lambda L_3^2$. Let $T\in \Aut(X)$  sending $L_1$ to $Z$, $L_2$ to $Y$ and $L_3$ to $X$. Therefore we can assume that $T$ sends $F$ to $YZ-X^2$. Since $P^1$ and $Q^1$ belong to $F^{K^{1/3}}$ then there exist $t,u\in K^{1/3}$ such that $T\otimes_X 1_Y$ sends $\{P^1,Q^1\}$ to $\{(t:t^2:1),(u: u^2:1)\}$ .
\end{proof}

\begin{prop}\label{primer isomorfismo loco}
Suppose that $L$ is a non $X$-invariant line of $Y$ defined over $K^{1/3}$ and let $L_1$ be the $X$-quotient of $L^3$. If $E$ is a one-dimensional degree $d$ subscheme of $X$ with intersection cycle  $$E^{\overline{K}} \boldsymbol{\cdot} L ^{\overline{K}}=\frac{n_1[\textbf{k}(E,P_1):K]}{3}P^{\overline{K}}_1+\ldots + \frac{n_s[\textbf{k}(E,P_s):K]}{3}P^{\overline{K}}_s$$ where $P_i$ is a non $K$-rational point for $i=1,\ldots,s$, then $\Ou_{L_1}(n_1P_1+\ldots+n_sP_s)\simeq \Ou_{L_1}(d)$.
\end{prop}
\begin{proof}
The hypothesis says that there exists a section $s\in H^0(L_1,\Ou_{L_1}(d))$ of $X$ defined by $E$. So the Weil divisor defined by $s$ in $L_1$ is  $n_1P_1+\ldots+n_sP_s$. Since $\Ou_{L_1}(d)$ is equal to $s\Ou_{L_1}(n_1P_1+\ldots+n_sP_s)$ then $\Ou_{L_1}(d)\simeq \Ou_{L_1}(n_1P_1+\ldots+n_sP_s)$.
\end{proof}

\begin{prop}\label{proposicion que ahora si resuelve el problema de los puntos}
If the quartic $E$ in $X$ is non-smooth at the points $P$ and $Q$ and if $P^1$ and $Q^1$ are $K^{1/3}$-rational points of $1$-type $1$ or $2$, then there are $u,t\in K^{1/3}$ such that:
\begin{itemize}
    \item The set $\{P^1,Q^1\}$ is $X$-equivalent to $\{(t:t^2:1),(u: u^2:1)\}$ if $P^1$ or $Q^1$ have $1$-type $2$.
    \item The set $\{P^1,Q^1\}$ is $X$-equivalent to $\{(t:0:1),(0:u:1)\}$ otherwise.
\end{itemize}   
\end{prop}
\begin{proof}
If $P^1$ and $Q^1$ have $1$-type $1$ there are $X$-invariant distinct lines $M_1,M_2$ in $Y$ containing $P^1$ and $Q^1$ respectively. If $M_3$ is any other $X$-invariant line ($M_3 \neq \overline{P^1Q^1}$), then by Proposition \ref{auto lines send} the set $\{M_1,M_2,M_3\}$ is $X$-equivalent to $\{x,y,z\}$. So $\{P^1,Q^1\}$ is $X$-equivalente to $\{(0:u:1),(t:0:1)\}$ with $u,t\in K^{1/3}$.

For the other case, let $L$ be the line through $P^1$ and $Q^1$. Since $P^1$ or $Q^1$ have $1$-type two then $L$ is not $X$-invariant. If $M_1$,$M_2$ are the lines in $Y$ defined over $K^{1/3}$, such that $M_1\cap L= P^1$ and $M_2\cap L= Q^1$, then we have that $L^{\overline{K}} \boldsymbol{\cdot} (F)^{\overline{K}}= 3P^{\overline{K}}+3Q^{\overline{K}}$, where $F=M_1^3M_2^3$.
 
Notice that $[\textbf{k}(F,P):K]=[\textbf{k}(F,Q):K]=3$. Indeed since $P^1$ have $1$-type one or two then there is an $X$-invariant conic $C_1$, which is regular at $P^1$. So Proposition \ref{pre grado emb} for $Z=C_1$ implies that $\textbf{k}(F,P)$ is a simple extension of $K$, because $\textbf{k}(F,P)=\textbf{k}(X,P)=\textbf{k}(E,P)=\textbf{k}(C_1,P)$. On the other hand, from the inclusion $\textbf{k}(E,P)\subset \textbf{k}(E^{K^{1/3}},P^1)=K^{1/3}$ we get that $[\textbf{k}(F,P):K]=3$. Therefore, Proposition \ref{primer isomorfismo loco} applied to $F$ says that $\Ou_{L_1}(6)\simeq \Ou_{L_1}(3P+3Q)$. 
 
 On the other hand $E^{\overline{K}} \boldsymbol{\cdot} L^{\overline{K}}=m_1P^{\overline{K}}+m_2Q^{\overline{K}}$ where $m_1,m_2\geq 2$ because $E^{\overline{K}}$ is singular at  $P^{\overline{K}}$ and $Q^{\overline{K}}$. Using Bézuot's Theorem we have $m_1+m_2=4$ and, therefore, $E^{\overline{K}} \boldsymbol{\cdot} L^{\overline{K}}=2P^{\overline{K}}+2Q^{\overline{K}}$.  Hence, Proposition \ref{primer isomorfismo loco} applied to $E$ provides that $\Ou_{L_1}(4)\simeq  \Ou_{L_1}(2P+2Q)$. As a consequence of the two sheaf isomorphisms above we have that $\Ou_{L_1}(2)\simeq \Ou_{L_1}(P+Q)$ and we conclude the proof from Proposition \ref{segundo isomorfismo loco}.
\end{proof}

\section{The structure of arithmetic genus three regular curves with two non-smooth points} \label{section 6}

In this section we will consider $C$ as a regular, non-smooth, geometrically integral and complete curve of arithmetic genus $3$ over a separably closed field $K$ of characteristic $3$. 

Notice that the arithmetic genus $p_a(C^L)$ of $C^L$ is also equal to $3$ since arithmetic genus is invariant under base change. However the geometric genus $p_g(C^L)$ may be smaller than $3$ (because $C^L$ may be non-regular) and, in this chapter, we will classify such phenomenon when
\[
p_g(C^{\overline{K}})=1.\]
In this case, we say that $C$ is an  \emph{geometrically elliptic} curve of arithmetic genus $3$. 
Notice also that we always have $p_a(\widetilde{C^{\overline{K}}})=p_g(\widetilde{C^{\overline{K}}})=p_g(C^{\overline{K}}).$

We remember that the curve $C$, with the additional hypothesis $p_g(C^{\overline{K}})=1$,  admits exactly two non-smooth points of singularity degree one (see \cite{Salomon} Remark 4.1). If $P$ is a non-smooth point of $C$, then Corollary  \ref{estudio del grado de singularidad} together with the fact that the singularity degree is the number of gaps of $\Gamma_{C^{\overline{K}}, P^{\overline{K}}}$ (see \cite{Bedoya} Proposition 1.1) implies that $P$ admits degree $3$, $\Gamma_{C^{\overline{K}}, P^{\overline{K}}}$ is generated by $2$ and $3$ and $P$ has differential degree $2$. Moreover, it also follows from the same corollary that the points of $C_1$ lying under the non-smooth points of $C$ are rational and, hence, $C_1$ is smooth.

\begin{Remark}\label{escritura de curvas de genero 3 como cuarticas}
If $C^{\overline{K}}$ is a non-hyperelliptic curve then there exists a morphism from $C$ to $\Py^2_{K}$ such that $C$ is isomorphic to a quartic defined over $K$. Such a
 morphism from $C$ to $\Py^2_{K}$ will be called a canonical embedding of $C$.

Indeed, since $C$ is regular, then it is Gorenstein and, consequently, $C^{\overline{K}}$ is also Gorenstein because this notion is invariant under base change. Therefore, Theorem 3.3, pg. 124, in \cite{Stohr4} implies that the morphism induced by the canonical sheaf defines an immersion from $C^{\overline{K}}$ onto a quartic curve in $\Py^2_{\overline{K}}$. Thus $\omega_{C^{\overline{K}}/\overline{K}}$ is very ample and, using exercise 1.30, pg 177, in \cite{Liu}, we conclude the same to $\omega_{C/K}$. To complete the proof we just need to observe that the same exercise shows that the embedding of $C^{\overline{K}}$ is obtained from the embedding of $C$. 
\end{Remark}

\subsection{Extrinsic geometry of geometrically elliptic, regular curve of arithmetic genus three} \label{subsection 61}

This section will study the relations between the sheaves $\Omega_{C/C_1}$, $\omega_{C/K}$ and $\Ou_{C}(P+Q)$ where $P$ and $Q$ are the non-smooth points of $C$. Consequently, we will identify certain uniquely determined divisors of $C$ to reach its embedding in $\mathbb{P}^2_K$.  

\begin{prop}\label{lemam acerca de los divisores canonicos}
We have the following isomorphisms.
\begin{itemize}
    \item $\Omega_{C/C_1}^{\otimes (-2)}\simeq \omega_{C/K}$
    \item $\Omega_{C/C_1}^{\otimes (-1)}\otimes \Ou_C(P+Q)\simeq \omega_{C/K}^{\otimes 2}$
    \item $\Ou_C(2P+2Q)\simeq \omega_{C/K}^{\otimes 3}$
\end{itemize}
\end{prop}
\begin{proof}
    Since $p_a(\widetilde{C^1})=1$ implies the isomorphism   $\omega_{\widetilde{C^1}/K^{1/3}}\simeq \Ou_{\widetilde{C^1}}$, then Theorem \ref{segundo lema de divisores canonicos} says that $\mathfrak{b}_1^{\ast}(\Omega_{C/C_1})\simeq\Ou_{\widetilde{C^1}}(-\widetilde{P^1}-\widetilde{Q^1})$ and $\mathfrak{b}_1^{\ast}(\omega_{C/K})\simeq\Ou_{\widetilde{C^1}}(2\widetilde{P^1}+2\widetilde{Q^1})$. Therefore, we conclude the proof by using the isomorphism $\mathfrak{b}_1^{\ast}(\Ou_C(P+Q))\simeq\Ou_{\widetilde{C^1}}(3\widetilde{P^1}+3\widetilde{Q^1})$ together with Theorem \ref{Lema a cerca de igualdad de divisores}.
\end{proof}

\begin{prop}\label{Lema auxiliar para generar la seccion clave}
    $\dim_KH^{0}(C,\Omega_{C/C_1}^{\otimes(-1)})=1$
\end{prop}
\begin{proof} 

Since the non-smooth points of $C$ have degree three, then their residue fields are simple extensions of $K$ and we may use Proposition \ref{tipo 1 tipo 2 prueba} and Proposition \ref{proposicion que ahora si resuelve el problema de los puntos}  to conclude the existence of an invariant conic through $P^1$ and $Q^1$. In particular there exists a section \[s\in H^0\left(C^{K^{1/3}},\omega_{C^{K^{1/3}}/K^{1/3}}^{\otimes 2}\otimes \mathcal{I}_{C^{K^{1/3}},P^1}\otimes \mathcal{I}_{C^{K^{1/3}},Q^1}\right)^{\nabla(2)}=H^0\left(C,\omega_{C/K}^{\otimes 2}\otimes \Ou_{C}(-P-Q)\right)\] and we conclude, by Proposition \ref{lemam acerca de los divisores canonicos}, that $\dim_KH^{0}(C,\Omega_{C/C_1}^{\otimes (-1)})\geq 1$.

Notice that $\Omega_{C/C_1}^{\otimes(-1)}$ has degree two, from Proposition \ref{lemam acerca de los divisores canonicos}, and hence  $\mathfrak{b}_1^{\ast}(\Omega_{C/C_1}^{ \otimes (-1)})$ has also degree two. On the other hand, Riemann-Roch theorem implies that $\dim_{K^{1/3}} H^{0}(\widetilde{C^1},\mathfrak{b}_1^{\ast}(\Omega_{C/C_1}^{\otimes (-1)}))=2$ and, consequently $$\dim_{K^{1/3}} H^{0}(C^1,\mathfrak{e}_1^{\ast}(\Omega_{C/C_1}^{\otimes(-1)}))\leq 2.$$ If $\dim_{K^{1/3}} H^{0}(C^1,\mathfrak{e}_1^{\ast}(\Omega_{C/C_1}^{\otimes(-1)}))= 2$ then $\mathfrak{e}_{\infty}^{\ast}(\Omega_{C/C_1}^{\otimes(-1)})$ would induce a degree two mophism from $C^{\overline{K}}$ to $\Py^1_{\overline{K}}$, showing the hyperellipticity of $C^{\overline{K}}$, and so, contradicting our hypothesis. Therefore $$\dim_KH^{0}(C,\Omega_{C/C_1}^{\otimes(-1)})=\dim_{K^{1/3}} H^{0}(C^1,\mathfrak{e}_1^{\ast}(\Omega_{C/C_1}^{\otimes (-1)}))= 1.$$
\end{proof}

\begin{Remark}
\label{Existencia de seccion canonica para divisores efectivos}
We remember the following fact about algebraic geometry because we will often need it. Given an effective divisor $D$ of $C$, there is a section $s\in H^0(C,\Ou_C(D))$ such that $\Div(s)=D$.

Indeed, we may write $D=(f_i,U_i)_{i\in I}$ with $f_i\in \Ou_C(U_i)$ since $D$ is effective. Then $1\in f_i^{-1}\Ou_C(U_i)$ for every $i\in I$ and so $1\in H^0(C,\Ou_C(D))$. From the definition (see \cite{Liu}[Exercise 1.13 pg. 266]) we have immediately that $\Div(1)=D$.
\end{Remark}

\begin{prop}\label{proposicion de la existencia de la seccion especial}
 There exists a section $s\in H^0(C,\omega_{C/K}^{\otimes 2})$ satisfying $\Div(s)=R+P+Q$ where $R$ is a divisor of degree $2$.
\end{prop}
\begin{proof}
Let $t\in H^0(C,\Ou_C(P+Q))$ be a global section such that $\Div(t)=P+Q$. From the last two propositions there is a non-trivial section $t_1\in H^0(C,\omega_{C/K}^{\otimes 2}\otimes \Ou_C(-P-Q))$ and we will  denote the  effective divisor $\Div(t_1)$ by $R$. Using again \cite{Liu}[Exercise 1.13 pg. 266] we get $\Ou_C(R)\simeq \omega_{C/K}^{\otimes 2}\otimes \Ou_C(-P-Q)$, which has degree two, and then $R>0$.  Therefore $s=t_1t\in H^0(C,\omega_{C/K}^{\otimes 2})$ satisfies $\Div(s)=R+P+Q$. 
\end{proof}
\begin{lemm}\label{lemma introductorio para la construccion de nuestras conicas}
There exists a unique degree two effective divisor $R$ of $C$ such that $\Ou_C(R+P+Q)\simeq \omega_{C/K}^{\otimes 2}$ and $\Ou_C(2R)\simeq \omega_{C/K}$. In particular, $P,Q \not\in \Supp(R)$, which contains only rational points.
\end{lemm}
\begin{proof}
Proposition \ref{proposicion de la existencia de la seccion especial} says that there exists a section $s\in H^0(C,\omega_{C/K}^{\otimes 2})$  such that $\Div(s)=R+P+Q$ where $R$ is a degree two effective divisor of $C$. Using \cite{Liu}[Exercise 1.13 pg 266] we have $\omega_{C/K}^{\otimes 2}\simeq \Ou_C(R+P+Q)$. 
Proposition \ref{lemam acerca de los divisores canonicos} says $\Ou_C(2P+2Q)\simeq \omega_{C/K}^{\otimes 3}$, that is,  $\Ou_C(2R)\simeq \omega_{C/K}^{\otimes 4}\otimes \omega_{C/K}^{\otimes -3}\simeq \omega_{C/K}$. 

Now if $R'$ is a degree two divisor satisfying $\Ou_C(R'+P+Q)\simeq \omega_{C/K}^{\otimes 2}$ and $\Ou_C(2R')\simeq \omega_{C/K}$ then $R$ is linearly equivalent to $R'$ and we can write $R=R'+\Div(f)$ with $f\in K(C)$. Since $C^{\overline{K}}$ is non-hyperelliptic, we have that $H^0(C,\omega_{C/K})=H^0(C,\Ou_C(2R'))$ is a base point free linear system and hence $H^0(C,\Ou_C(R')) \subsetneq H^0(C,\Ou_C(2R'))$, that is, \[\dim_K H^0(C,\Ou_C(R'))< \dim_K H^0(C,\omega_{C/K})=3.\]
On the other hand, if $f\not \in K$ then the inclusion 
\[K \oplus fK \subseteq H^0(C,\Ou_C(R'))\]
would provides that $\dim_K H^0(C,\Ou_C(R'))=2$. However, since $R'$ is a degree two divisor, it would imply that $C$ and  $C^{\overline{K}}$ are hyperelliptic curves, which is a contradiction. Therefore, $f\in K$ and $R=R'$.
\end{proof}

\begin{prop}\label{Ecuacion de C en genero 1}
There are $F\in H^0(C,\omega_{C/K}^{\otimes 2})$, $L_1\in H^0(C,\omega_{C/K})$ and $L \in H^0(C^1,\mathfrak{e}_1^{\ast}(\omega_{C/K})) \setminus H^0(C,\omega_{C/K})$ satisfying $\Div(F)=R+P+Q$ (where $R$ is the divisor constructed in Lemma \ref{lemma introductorio para la construccion de nuestras conicas}), $\Div(L_1)=2R$ and $\Div(L^3)=2P+2Q$. Moreover, $C$ is defined by  $L_1L^3-F^2$ as a curve embedded in $\Py^2_{K}$.
\end{prop}
\begin{proof}The existence of $F$ and $L_1$ are obtained by the last two results and by the last remark. If we take the section $L$ in $H^0(C^1,\mathfrak{e}_1^{\ast}(\omega_{C/K}))$ passing through the $K^{1/3}$-rational points $P^1$ and $Q^1$, then $\mathfrak{n}_1^{\ast}\Div(L^3)=6\widetilde{P^1}+6\widetilde{Q^1}$. Proposition \ref{lemam acerca de los divisores canonicos} says that there is $L'\in H^{0}(C,\omega_{C/K}^{\otimes 3})$ such that $\Div(L')=2P+2Q$. Since $\Div((\mathfrak{n}_{1})^{\ast}(L^3))=(\mathfrak{b}_{1})^{\ast}(\Div(L'))$ then $L'=aL^3$, for some $a\in K^{1/3}$, and $\Div(L^3)=\Div(L')$ as divisors of $C$. We also have that $L \not\in H^0(C,\omega_{C/K})$ because, otherwise, we would have $P, Q \in L$, which in turn implies the  inequality $4=\deg(\omega_{C/K})=\deg(\Div (L))\geq \deg(P)+\deg(Q) =6$.

Notice that $L_1L^3$ and $F^2$ define the same divisor of $C$. Hence, we conclude that there exist $c\in K\setminus\{0\}$ such that the residual class of $L_1L^3-cF^2$ vanishes in $H^0(C,\omega^{\otimes 4}_{C/K})$. Moreover, in $H^0(C,\omega_{C/K})^{\otimes 4}$, we have $L_1L^3\neq bF^2$ for every $b\in K\setminus\{0\}$. Indeed if we would have $L_1L^3= bF^2$ then $L_1=b'L$ for some $b'\in K\setminus\{0\}$ showing that $\Div(L)= \Div(L_1)$ as divisors of $C^1$. However this is a contradiction because $P,Q \not\in \Supp(R)$. Therefore $L_1L^3-cF^2$ is the equation defining $C$ in $\Py^2_{K}$. After multiplying $L_1$ by $c^{-1}$ we can assume $c=1$.
\end{proof}

\subsection{The parameter space of geometrically elliptic regular curves of genus three}\label{subsection 5.2}

Now we will construct the parameter space for geometrically elliptic regular curves of arithmetic genus three and, to do this, we keep using the notation of the last subsection. In this way, we will use Proposition \ref{Ecuacion de C en genero 1} to give an $\mathbb{P}^2_K$-equivalent equation for $C$ and, indeed, we also show that this equation induces a regular, geometrically elliptic curve of genus three. After that, we will construct the parameter space.

\begin{thm}\label{Escritura generica de las familias de cuarticas que nos interesan}
If $x,y,z$ is a basis of $H^0(\Py^2_K,\Ou_{\Py^2_K}(1))$ as a $K$-vector space,  then there exists a non $\mathbb{P}^2_K$-invariant line $L$ in $H^0(\Py^2_{K^{1/3}},\Ou_{\Py^2_{K^{1/3}}}(1))$ such that:
\begin{itemize}
    \item The curve $C$ is $\Py^2_K$-equivalent to a quartic defined by the polynomial $xL^3-(x^2-yz)^2$ if $R$ is reduced and $P^1$ and $Q^1$ have $1$-type $2$.
\item The curve $C$ is $\Py^2_K$-equivalent to a quartic defined by the polynomial $zL^3-(x^2-yz)^2$ if $R$  is not reduced and $P^1$ and $Q^1$ have $1$-type $2$.
\item The curve $C$ is $\Py^2_K$-equivalent to a quartic defined by the polynomial $zL^3-(xy)^2$ if $P^1$ and $Q^1$ have $1$-type $1$.
\end{itemize}
where $R$ is the divisor defined in Lemma \ref{lemma introductorio para la construccion de nuestras conicas}. Moreover $P^1$ and $Q^1$ are the two points of intersection between $L$ and $F$ where $F=x^2-yz$ in the first two cases and $F=xy$ in the last one.

\end{thm}
\begin{proof}
Let $F$ and $L_1$ be the $\Py^2_K$-invariant conic and $\Py^2_K$-invariant line defined in Proposition \ref{Ecuacion de C en genero 1} respectively. 

Suppose that $F$ is irreducible over $K$. Since $F$ contains $P^1$ and $Q^1$ we may conclude from Corollary \ref{no existen lineas pasando por puntos tipo 2} that there is no  $\Py^2_K$-invariant line passing through $P^1$ and $Q^1$, that is, both points have $1$-type $2$. Since $K$ is separably closed and $R$ has degree two, then the support of $R$ has a $K$-rational point and, in particular, the genus zero curve induced by $F$ and the genus zero curve induced by $x^2-yz$ are $\Py^2_K$-equivalent.  Moreover if $R$ is reduced (respectively if $R$ is non-reduced) we can choose a $\Py^2_K$-invariant automorphism $\psi$ (meaning an automorphism of $\Py^2_{K^{1/3}}$ with coefficients in $K$) sending $F$ to $x^2-yz$ such that $\psi(L_1)$ is $x$ (respectively $z$). Therefore $\psi$ is the $\Py^2_{K^{1/3}}$-automorphism that proves this theorem for the case where $F$ is irreducible.

Finally if $F$ is reducible over $K$ then $P^1$ and $Q^1$ have $1$-type $1$. Since $F$ is a $\Py^2_K$-invariant conic passing through $P^1$ and $Q^1$ then $F=L_2L_3$ where $L_2$ and $L_3$ are the $\Py^2_K$-invariant lines. We notice that we can not have both points $P^1$ and $Q^1$  belonging to $L_i$ ($i=2$ or $3$) because, otherwise, $4=\deg(\Div(L_i))\geq \deg(P + Q) = 6$. In this way we can assume that $P^1 \in L_2$ and $Q^1\in L_3$. Notice also that $L_1$ is different from $L_2$ and $L_3$, because $C$ is irreducible, and $L_2\neq L_3$, because the curve given by $L_1L^3-L_2^4$ admits $L\cap L_2$ as its unique non-smooth point. Now we just need to take $\psi$ as the $\Py^2_K$-invariant automorphism sending $L_1,L_2,L_3$ to $z,x,y$, respectively, to finish the proof.  \end{proof}

Now we prove a kind of converse of the above theorem, by showing that a curve in $\Py^2_{K^{1/3}}$ whose equation is given by one of those three possibilities, is a base extension of a non-smooth regular curve in $\Py^2_{K}$ of genus $3$ with two non-smooth points.

\begin{prop}\label{regularidad}
Let $E$ be a geometrically integral curve in $\Py^2_{K}$ defined by the polynomial $L^3L_1-F^2$ where $F \in K[x,y,z]$ is a conic and $L  \in K^{1/3}[x,y,z]$ and $L_1 \in K[x,y,z]$ are lines. Suppose that the intersection between $F$ and $L$ in $\Py^2_{K^{1/3}}$ is given by two distinct points $P^1$ and  $Q^1$ and suppose that $P, Q \in \Py^2_{K}$ are the points lying under $P^1$ and $Q^1$ , respectively. If the unique non-smooth points of $E$ are $P$ and $Q$ and if $P^1$ and $Q^1$ are non $\Py^2_K$-invariant points then $E$ is regular.
\end{prop}
\begin{proof}
From Proposition \ref{interseccion de curvas invariantes}, $P^1$ and $Q^1$ do not lie in the intersection between $L_1$ and $F$. Hence  the intersection indexes between $E^{\overline{K}}$ and $F^{\overline{K}}$ at $P^{\overline{K}}$ and $Q^{\overline{K}}$ are equal to $3$. On the other hand, $F \in \m_{E,P}$ and $P^1$ non $\Py^2_K$-invariant imply, respectively, that $\deg P \leq \dim_{\overline{K}}\Ou_{E^{\overline{K}},P^{\overline{K}}}/<F> =3$ and $\deg P>1$, that is, $\deg P=3$. Analogously $\deg Q=3.$ Using Corollary \ref{Criterio para la regularidad de puntos} for the ideal sheaf $\mathcal{I}_{F}$, we have that $E$ is a regular curve.  
\end{proof}

\begin{defn}
Fixing a basis $x,y,z$ of $H^0(\Py^2_K,\Ou_{\Py^2_K}(1))$ we define the following three families of quartics.
\begin{itemize}
\item[{\bf$\mathcal{C}_0 \, :$}] $a(x^2-yz)^2+((t_1^2-t_2^2)x+(t_2-t_1)y+(t_1t_2^2-t_2t_1^2)z)^3x$ where $t_1\neq t_2\in K^{1/3}\setminus K$ and $a\in K\setminus\{-(t_1^2-t_2^2)^3\}$.
\item[{\bf$\mathcal{C}_1 \, :$}] $a(x^2-yz)^2+((t_1^2-t_2^2)x+(t_2-t_1)y+(t_1t_2^2-t_2t_1^2)z)^3z$ where $t_1\neq t_2\in K^{1/3}\setminus K$ and $a\in K$.
\item[{\bf$\mathcal{C}_2 \, :$}]
$a(xy)^2+(t_1x+t_2y-t_1t_2z)^3z$ where $t_1\neq t_2\in K^{1/3}\setminus K$  and $a\in K$.
\end{itemize}
We say that $t_1,t_2$ and $a$ are the parameters of $C$.

\end{defn}

\begin{thm}\label{tercer Teorema central de este trabajo}
If $C\in \mathcal{C}_0\cup \mathcal{C}_1\cup \mathcal{C}_2$, then $C$ is a regular curve of arithmetic genus three such that $C^{\overline{K}}$ is a non-hyperelliptic and integral curve with normalization having genus one.
\end{thm}
\begin{proof}
Let $C\in \mathcal{C}_0\cup\mathcal{C}_1\cup\mathcal{C}_2$ and take the parameters $t_1,t_2,a$ of $C$. We define  $P^{1}=(t_1:t_1^2:1)$ and $Q^1=(t_2:t_2^2:1)$  if $C\in \mathcal{C}_0\cup\mathcal{C}_1$ or $P^{1}= (0:t_1:1)$ and $Q^1=(t_2:0:1)$ if $C \in \mathcal{C}_2$. Let $P$ and $Q$ be the points of $C$ lying under $P^1$ and $Q^1$ respectively. After applying the Jacobian criterion for the curve $C^{\overline{K}}$ we obtain that the non-smooth points of $C$, in both cases, are  $P$ and $Q$. 

Notice that the singular points of $C^{\overline{K}}$ are  cuspidal points and, therefore, $C^{\overline{K}}$ is irreducible. Since $P^1$ and $Q^1$ are non $\Py^2_K$-invariant points, Proposition \ref{regularidad} says that $C$ is regular. Moreover since $C^{\overline{K}}$ is a plane quartic curve then $C^{\overline{K}}$ is non-hyperelliptic of genus three. Finally $\widetilde{C^{\overline{K}}}$ has genus one from \cite{Salomon} Remark 4.1. 
\end{proof}

\begin{thm}\label{Primer Teorema central de este trabajo}
If $C$ is a regular curve over $K$ of arithmetic genus three such that $C^{\overline{K}}$ is a non-hyperelliptic and integral curve with normalization having genus one, then there is an embedding of $C$ in $\Py^2_K$ such that the image of $C$ is a curve in $ \mathcal{C}_0\cup\mathcal{C}_1\cup\mathcal{C}_2$.
\end{thm}
\begin{proof}
Corollary \ref{estudio del grado de singularidad} and Proposition \ref{tipo 1 tipo 2 prueba} say that the singular points $P^1$ and $Q^1$ of $C^1$, lying over the non-smooth points  $P$ and $Q$ of 
$C$, respectively, are $K^{1/3}$-rational points of $1$-type one or two. On the other hand, after apply the canonical embedding of $C$, Theorem \ref{Escritura generica de las familias de cuarticas que nos interesan} implies that:
\begin{itemize}
    \item $C$ is $\Py^2_K$-equivalent to a quartic defined by the polynomial $xL^3-(x^2-yz)^2$ if $R$ is reduced and $P^1$ and $Q^1$ have $1$-type $2$.
\item  $C$ is $\Py^2_K$-equivalent to a quartic defined by the polynomial $zL^3-(x^2-yz)^2$ if $R$  is not reduced and $P^1$ and $Q^1$ have $1$-type $2$.
\item  $C$ is $\Py^2_K$-equivalent to a quartic defined by the polynomial $zL^3-(xy)^2$ if $P^1$ and $Q^1$ have $1$-type $1$.
\end{itemize}
Moreover $P^1$ and $Q^1$ are the two points of intersection between $L$ and $F$ where $F=x^2-yz$ in the first two cases and $F=xy$ in the last one.

We notice that in every case the non $\Py^2_K$-invariance at $P^1$ and $Q^1$ implies that they do not belong to the line given by $z$. Hence, in the first two cases $P^1=(t_1:t_1^2:1)$ and $Q^1=(t_2:t_2^2:1)$  and, in the third one, $P^1=(0:t_1:1)$ and $Q^1=(t_2:0:1)$, with $t_1\neq t_2\in K^{1/3}\setminus K$. Therefore, 
\begin{itemize}
    \item $C$ is $\Py^2_K$-equivalent to a quartic defined by the polynomial $a(x^2-yz)^2+((t_1^2-t_2^2)x+(t_2-t_1)y+(t_1t_2^2-t_2t_1^2)z)^3x$ if  $R$ is reduced and $P^1,Q^1$ have $1$-type $2$.
\item $C$ is $\Py^2_K$-equivalent to a quartic defined by the polynomial $a(x^2-yz)^2+((t_1^2-t_2^2)x+(t_2-t_1)y+(t_1t_2^2-t_2t_1^2)z)^3z$ if $R$  is not reduced and $P^1,Q^1$ have $1$-type $2$.
\item $C$ is $\Py^2_K$-equivalent to a quartic defined by the polynomial $(xy)^2+(t_1x+t_2y+t_1t_2z^3)z$ if $P^{1},Q^{1}$ have $1$-type $1$.
\end{itemize}
where $a\in  K\setminus\{-(t_1^2-t_2^2)^3\}$  in the first case and $a\in K$ in the last two cases.
\end{proof}
\begin{Remark}\label{remmark que dice las relaciones y a}
From the proof of the above theorem, every curve in $\mathcal{C}_0$ (respectively, in $\mathcal{C}_1$) can be written by the equation $(x^2-yz)^2+(Ax+By+Cz)^3x$ (respectively, by  $(x^2-yz)^2+(Ax+By+Cz)^3z$)  where $A=(t_1^2-t_2^2)/a^{1/3}$, $B=(t_2-t_1)/a^{1/3}$ and $C=(t_1t_2^2-t_2t_1^2)/a^{1/3}$. Moreover, every curve in $\mathcal{C}_2$ can be written by the equation $(xy)^2+(Ax+By+Cz)^3z$ where $A=t_1/a^{1/3},B=t_2/a^{1/3}$ and $C=-t_1t_2/a^{1/3}$.
\end{Remark}

\begin{thm}\label{tabla de isomorfismos}
There is no curve in $\mathcal{C}_i$ which is $\Py^2_K$-equivalent to a curve in $\mathcal{C}_j$ for $i\neq j$. Moreover we have that:

\begin{enumerate}
    \item If $C,C'\in \mathcal{C}_0$ are $\Py^2_K$-equivalent and $C$ is defined by the equation  
    $(x^2-yz)^2+(Ax+By+Cz)^3x$ where $A,B,C\in K^{1/3}$
    then $C'$ is defined by one of the equations 
\[(x^2-yz)^2+( Ax+\beta\lambda By+\beta\lambda^{-1}Cz)^3x 
\text{ or } 
(x^2-yz)^2+(Ax+\lambda\beta Cy+\lambda^{-1}\beta Bz)^3x\]
 where $\lambda\in K\setminus\{0 \}$ and $\beta\in \{1,-1\}$. The automorphism $T$ of $\Py^2_K$ sending $C$ to $C'$ is defined by  $T(x:y:z)=(\beta x: \lambda y: \lambda^{-1}z)$ or $T(x:y:z)=(\beta x: \lambda^{-1} z: \lambda y)$ 
respectively.

\item If $C,C'\in \mathcal{C}_1$  are $\Py^2_K$-equivalent and $C$ is defined by the equation  
    $(x^2-yz)^2+(Ax+By+Cz)^3z$ where $A,B,C\in K^{1/3}$,
    then $C'$ is defined by the equation:
\[(x^2-yz)^2+\left( \left(\frac{\beta A+B\lambda_1}{\lambda_2^{1/3}}\right)x+\lambda_2^{2/3}B y+
\left( \frac{
B\lambda_1^2-\beta A\lambda_1+C}{\lambda_2^{4/3}}
\right)z\right)^3z\]   where $\lambda_1\in K$, $\lambda_2\in K\setminus\{0\}$ and $\beta\in \{1,-1\}$.  The automorphism $T$ of $\Py^2_K$ sending $C$ to $C'$ is defined by $T(x:y:z)=(\beta x - \beta(\lambda_1/\lambda_2)z: \lambda_1 x+ \lambda_2 y+ (\lambda_1^2/\lambda_2)z: \lambda_2^{-1} z)$. 

\item  If $C,C'\in \mathcal{C}_2$  are $\Py^2_K$-equivalent and $C$ is defined by the equation  
$(xy)^2+(Ax+By+Cz)^3z$ where $A,B,C\in K^{1/3}$,
    then $C'$ is defined by one of the equations:
\[(xy)^2+(\lambda_2^{1/3}\lambda_1Ax+\lambda_2^{1/3}\lambda_1^{-1}By+\lambda_2^{4/3}Cz)^3z  
\text{ or }
(xy)^2+(\lambda_2^{1/3}\lambda_1Ay+\lambda_2^{1/3}\lambda_1^{-1}Bx+\lambda_2^{4/3}Cz)^3z\]
where $\lambda_1,\lambda_2\in K\setminus\{0\}$. The automorphism $T$ of $\Py^2_K$ sending $C$ to $C'$ is defined by $T(x:y:z)=(\lambda_1 x: \lambda_1^{-1} y: \lambda_2 z)$ or $T(x:y:z)=(\lambda_1 y: \lambda_1^{-1} x: \lambda_2 z)$ respectively.
\end{enumerate}

\end{thm}
\begin{proof}
Suppose that two curves $C,C'$ in  $\mathcal{C}_0\cup \mathcal{C}_1\cup C_2$ are $\Py^2_K$-equivalent and let $P_C$, $Q_C$ and $P_{C'}$, $Q_{C'}$ be the non-smooth points of $C$ and $C'$ respectively. Moreover let $R_{C}$ and $R_{C'}$ be the divisors defined in Lemma \ref{lemma introductorio para la construccion de nuestras conicas} associated to $C$ and $C'$ respectively. From the proof of Theorem \ref{Primer Teorema central de este trabajo} we have that:
\begin{itemize}
    \item $C$ belongs to $\mathcal{C}_0$ if and only if $P_{C},Q_{C}$ have $1$-type $2$ and $R_C$ is reduced. 
    \item $C$ belongs to $\mathcal{C}_1$ if and only if $P_{C},Q_{C}$ have $1$-type $2$ and $R_C$ is not reduced.
    \item $C$ belongs to $\mathcal{C}_2$ if and only if $P_{C},Q_{C}$ have $1$-type $1$.
\end{itemize}
Since isomorphisms of curves preserve the linear system associated to the canonical sheaf, since the divisor $R_C$ is uniquely determined (cf. Lemma \ref{lemma introductorio para la construccion de nuestras conicas}) and since the definition of $1$-type and the reduceness is invariant under automorphisms of $\Py^2_K$,  there is no curve in $\mathcal{C}_i$ which is $\Py^2_K$-equivalent to a curve in $\mathcal{C}_j$ for $i\neq j$.

\begin{enumerate}

\item If $C\in \mathcal{C}_0$ is defined by the equation $(x^2-yz)^2+(Ax+By+Cz)^3x$ where $A,B,C\in K^{1/3}$ then $\Div(x^2-yz)=R_C+P_C+Q_C$ and $\Div(x)=2R_C$. Then the discussion above implies that the conic $x^2-yz$ and the line $x$ must be invariant under $T$.
Therefore $T$ is defined by $T(x:y:z)=(\beta x: \lambda y: \lambda^{-1}z)$ or $T(x:y:z)=(\beta x: \lambda^{-1} z: \lambda y)$ 
  where $\lambda\in K\setminus\{0\}$ and $\beta\in \{1,-1\}$. Then $C'$ is respectively defined by
\[(x^2-yz)^2+(Ax+\lambda\beta By+\lambda^{-1}\beta Cz)^3x  
\text{ or }
(x^2-yz)^2+(Ax+\lambda\beta Cy+\lambda^{-1}\beta Bz)^3x.\]

\item If $C\in \mathcal{C}_1$ is defined by the equation $(x^2-yz)^2+(Ax+B y+Cz)^3z$ where $A,B,C\in K^{1/3}$ then $\Div(x^2-yz)=R_C+P_C+Q_C$ and $\Div(z)=2R_C$. Then the above discussion implies that the conic $x^2-y,z$ and the line $z$ must be invariant under $T$. Therefore $T$ is defined by $T(x:y:z)=(\beta x - \beta(\lambda_1/\lambda_2)z: \lambda_1 x+ \lambda_2 y+ (\lambda_1^2/\lambda_2)z: \lambda_2^{-1} z)$
where $\lambda_1\in K$, $\lambda_2\in K\setminus\{0\}$ and $\beta\in \{1,-1\}$. Then $C'$ is defined by
\[(x^2-yz)^2+\left( \left(\frac{\beta A+B\lambda_1}{\lambda_2^{1/3}}\right)x+\lambda_2^{2/3}B y+
\left( \frac{
B\lambda_1^2-\beta A\lambda_1+C}{\lambda_2^{4/3}}
\right)z\right)^3z.\]

\item If $C\in \mathcal{C}_2$ is defined by the equation $(xy)^2+(Ax+By+Cz)^3z$ where $A,B,C\in K^{1/3}$ then $\Div(xy)=R_C+P_C+Q_C$ and $\Div(z)=2R_C$. Then the above discussion implies that the conic $xy$ and the line $z$ must be invariant under $T$.
Therefore $T$ is defined by  $T(x:y:z)=(\lambda_1 x: \lambda_1^{-1} y: \lambda_2 z)$ or $T(x:y:z)=(\lambda_1 y: \lambda_1^{-1} x: \lambda_2 z)$
where $\lambda_1,\lambda_2\in K\setminus\{0\}$. Then $C'$ is respectively defined by
\[(xy)^2+(\lambda_2^{1/3}\lambda_1Ax+\lambda_2^{1/3}\lambda_1^{-1}By+\lambda_2^{4/3}Cz)^3z
\text{ or }
(xy)^2+(\lambda_2^{1/3}\lambda_1Ay+\lambda_2^{1/3}\lambda_1^{-1}Bx+\lambda_2^{4/3}Cz)^3z.\]
\end{enumerate}
\end{proof}

\bibliography{referencias}

\noindent{\scriptsize\sc Universidade Federal Fluminense, Instituto de Matem\'atica e Estat\'istica.\\
Rua Alexandre Moura 8, S\~ao Domingos, 24210-200 Niter\'oi RJ,
Brazil.}

\vskip0.1cm

{\scriptsize\sl E-mail address: \small\verb?gborrelli@id.uff.br?}

{\scriptsize\sl E-mail address: \small\verb?cammat1985@gmail.com?}

{\scriptsize\sl E-mail address: \small\verb?rsalomao@id.uff.br?}

\end{document}